\documentclass[10pt]{article}

\usepackage{preamble}

\begin{document}

\newcommand\restr[2]{{
  \left.\kern-\nulldelimiterspace 
  #1 
  \vphantom{\big|} 
  \right|_{#2} 
  }}

\makeatletter
\renewcommand{\@seccntformat}[1]{%
  \ifcsname prefix@#1\endcsname
    \csname prefix@#1\endcsname
  \else
    \csname the#1\endcsname\quad
  \fi}
\makeatother

\makeatletter
\newcommand{\colim@}[2]{%
  \vtop{\m@th\ialign{##\cr
    \hfil$#1\operator@font colim$\hfil\cr
    \noalign{\nointerlineskip\kern1.5\ex@}#2\cr
    \noalign{\nointerlineskip\kern-\ex@}\cr}}%
}
\newcommand{\colim}{%
  \mathop{\mathpalette\colim@{\rightarrowfill@\textstyle}}\nmlimits@
}
\makeatother

\newcommand\rightthreearrow{%
        \mathrel{\vcenter{\mathsurround0pt
                \ialign{##\crcr
                        \noalign{\nointerlineskip}$\rightarrow$\crcr
                        \noalign{\nointerlineskip}$\rightarrow$\crcr
                        \noalign{\nointerlineskip}$\rightarrow$\crcr
                }%
        }}%
}

\title{The Spectrum of Units of Algebraic $K$-theory}         
\author{Shachar Carmeli and Kiran Luecke}        
\date{\today}          
\maketitle

\begin{abstract}
It is well known that the $[0,1]$ and $[0,2]$ Postnikov truncations of the units of the topological $K$-theories $\glone \KO$ and $\glone \KU$, respectively, are split, and that the splitting is provided by the ($\Z/2$-graded) line bundles. In this paper we give a similar splitting for the $[0,1]$-truncation of the units of algebraic $K$-theory, considered as a sheaf on affine schemes. A crucial step is to produce the splitting for $\glone K(\Z)$. Along the way we also give a complete calculation of the connective spectrum of strict units of $K(\Z)$ and $K(\F_\ell)$ for a prime $\ell$. Finally, we show that the units of algebraic $K$-theory do not split as a presheaf. In fact we show they do not even split pointwise.
\end{abstract}

\tableofcontents

\section{Introduction and outline}
\subsection{Background and motivation}
For a commutative ring spectrum $R$, one can associate a \emph{spectrum of units} $\glone R$, which plays a role analogous to that of the group of units of a classical commutative ring. The homotopy groups of $\glone R$ and $R$ are closely related---we have
\[
\pi_0 \glone R \simeq \pi_0 R^\times \quad  \text{and} \quad \pi_i \glone R \simeq \pi_i R \text{ for }  i>0.
\]
However, the nature of the spectrum $\glone R$, which consist of not only the homotopy groups but also the $k$-invariants gluing them together, is still a far from understood in general. 

Algebraic $K$-theory is a vast source of commutative ring spectra. For every classical commutative ring $S$ there is an associated commutative ring spectrum $K(S)$, the $K$-theory spectrum of $S$. It is thus natural to consider the spectrum $\glone K(S)$. 

For real and complex \emph{topological} $K$-theory, the spectra of units are well-understood (cf. e.g. \cite{adamspriddy}, \cite{blm}, \cite{mayetal}). For a spectrum $X$ denote by $X[a,b]$ (resp. $X[a,\infty)$) the $b$-th truncation of the $a$-th connective cover of $X$ (resp. the $a$-th connective cover of $X$).
 Then the spectrum $\glone \KU$ decomposes as the direct sum 
\[
\glone \KU \simeq \glone \KU[0,2] \oplus \glone \KU[3,\infty),
\]
and there is a $p$-completed isomorphism $\glone \KU[3,\infty) \simeq \KU[3,\infty)$. 
Similarly, for the real topological $K$-theory spectrum $\KO$ there is a splitting 
\[
\glone \KO \simeq \glone \KO[0,1] \oplus \glone \KO[2,\infty), 
\]
and a $p$-complete isomorphism 
$
\glone \KO [2,\infty) \simeq \KO [2,\infty).  
$
Thus, knowing the additive structure of these ring spectra and their spectra of units \emph{in a small range of degrees} gives all the information needed to determine their spectra of units as a whole, at least after completion at a prime.

\subsection{Main results}
The primary goal of this paper is to demonstrate that a similar phenomenon occurs also in the context of algebraic $K$-theory. Namely, we show that this is the case for the Zariski sheaf $\glone K$ that assigns to each commutative ring $S$ the connective spectrum $\glone K(S)$ (cf. \Cref{sectionsheafsplit}).

\begin{introthm} [\ref{sheafsplit_main}]\label{sheafsplitintro}
    The Zariski sheaf $\glone K$ is split at 1. That is, in the category of Zariski sheaves of connective spectra on affine schemes there is an isomorphism
    \[ \glone K\simeq \glone K[0,1]\oplus \glone K[2,\infty).
    \]
\end{introthm}
Because Postnikov truncation of sheaves involves pointwise truncation \emph{and} sheafification, this \emph{does not} imply that $\glone K(S)$ splits in a similar way for every ring $S$. In fact, we show that such a pointwise splitting does not exist.

\begin{introthm} [\ref{pointwisecounterexample}]
    The sheaf $\glone K$ does \emph{not} split pointwise. More precisely, for the ring $S = \R[x,y]/(x^2 + y^2 -1)$ we have 
    \[
    \glone K(S) 
\not{\simeq}\glone K(S)[0,1] 
\oplus \glone K(S)[2,\infty).
\]
\end{introthm}

We prove \Cref{sheafsplitintro} by reducing the problem to the splitting of $\glone K(\Z)$, in which case we also obtain information about the two summands.

\begin{introthm}[\ref{glintsplittingpcomplete}]\label{intro-splitting_glone} 
We have 
\[
\glone K(\Z) \simeq \glone K(\Z)[0,1] \oplus \glone K(\Z)[2,\infty),
\]
where the first summand is isomorphic to the fiber of the map 
\[
\Z/2 \oto{\Sq^2} \Sigma^2 \Z/2
\]
and after completion at a prime $p$ there is an isomorphism of spectra 
\[
\glone K(\Z)_p[2,\infty) \simeq K(\Z)_p[2,\infty). 
\]
\end{introthm}


Using the splitting of the sheaf $\glone K$, we obtain similar splittings for individual rings $S$, under some assumptions on their low $K$-groups.  For a commutative ring $S$, let $\hPic(S)$ be the Picard group of the category of discrete $S$-modules. Recall that there is a natural group homomorphism $\Z \oplus \hPic(S) \iso K_0(S)$ taking $(n,L)$ to (the $K_0$-class of) $S^{\oplus (n-1)}\oplus L$ and a natural homomorphism $S^\times \to K_1(S)$ taking $a\in S^\times$ to the loop at $[S]\in K(S)$ given by the automorphism $a\cdot \colon S\iso S$.    
\begin{introthm}[{{\ref{low_K_correct_splitting}}}]\label{intro-glone_general}
Let $S$ be a ring for which $\Z \oplus \hPic(S) \iso K_0(S)$ and $S^\times \iso K_1(S)$. Then 
\[
\glone K(S) \simeq \glone K(S)[0,1] \oplus \glone K(S)[2,\infty). 
\]
\end{introthm}


Closely related to the spectrum of units $\glone R$ is the connective spectrum of \emph{strict units} $\G_m R := \hom_{\Sp^{\cn}}(\Z,\glone R)$. In fact, we \emph{use} knowledge of the strict units to establish splittings of $\glone R$ in certain cases. We give a complete calculation of $\G_mK(S)$ for $S= \Z$ or a prime field $\F_\ell$. Let $\A=\Q\otimes\prod_p\Z_p\subseteq \prod_p \Q_p$ be the ring of finite Adeles. 

\begin{introthm}
[\ref{strictunitsKZ}, \ref{strictunitsKFl}]\label{intro-strict} There are isomorphisms 
    \[
    \G_mK(\Z)\simeq\widehat{\Z}\oplus \Z/2\oplus\Sigma(\Z/2)\bigoplus_{k\geq1}\Sigma^{4k}(\A/\Q),
    \]
\[\G_mK(\F_\ell)\simeq\F_\ell^\times\oplus\Sigma\F_\ell^\times.
    \]
\end{introthm}

\subsection{Methods}
Our method is to reduce the splitting of the Zariski sheaf $\glone K$ to the splitting of its global sections $\glone K(\Z)$. For this reduction we use the canonical map $\pic \to \glone K$ that sends an invertible module spectrum over $S$ to its $K$-theory class. To illustrate our method in a simpler case than $\Z$, we simultaneously treat the prime fields 
$\F_\ell$.  

Then, we reduce questions about $\glone K(\Z)$ and $\glone K(\F_\ell)$ to questions about their $T(1)$-localizations (at an implicit prime $p$). The main tool in this reduction is the Lichtenbaum-Quillen theorem for these rings, stating that their $p$-completed $K$-theory spectra diverge from the connective covers of their $T(1)$-localizations only in small degrees; for the rings in question the difference can be analyzed explicitly using knowledge of their low degree $K$-groups.

Once we are in the $T(1)$-local setting, we have new tools to study units and strict units. Namely, the strict units are governed by the power operation $\theta$ from \cite{hodgkin}, the spectrum of units $\glone R$ is related to $R$ itself by the Rezk logarithm map $\log \colon \glone R \to R$ \cite{rezk} which is an isomorphism in large degrees, and these two operations are related by nontrivial formulas. 
We use these tools to access the homotopy type of $\glone R$ and construct the desired splittings of their Postnikov towers. We remark that there is an alternative approach to studying $\glone K(\F_\ell)$, due to May (\cite[Section VIII.4]{mayetal}). 

One difficulty in passing from $R:= K(S)$ to $\LKone R$ is that the Postnikov tower of the latter \emph{does not split}. The reason is that $\pi_0\LKone R$, which in essentially all examples considered here is $ \Z_p$, has too many units. For this reason, we work with a variant of $\glone R$ which manually omits the extra units formed by $p$-completion or $T(1)$-localization (cf. \Cref{glintdef}).








\subsection{Acknowledgements}

We thank John Rognes for many helpful conversations, especially for explaining the proof of \Cref{LQprop}. We thank two anonymous referees for helpful comments and corrections.
S.C. is supported by a research grant from the Center for New Scientists at the Weizmann Institute of Science, and would like to thank the Azrieli Foundation for their support through an Early Career Faculty Fellowship.




    
    

\section{Preliminaries for the splitting of $\glone R$}
In this section we discuss some general reduction steps towards splitting the Postnikov tower of $\glone R$. As mentioned above, it will be convenient to work with the following variant of $\glone$. 
\begin{defn}\label{glintdef}
    For any commutative ring spectrum $E$ define $\glone ^{\pm1}E$ to be the pullback
    $$\begin{tikzcd}
        \glint E \arrow[r]\arrow[d]&\glone E\arrow[d] \\
        \{\pm1\}\arrow[r]&\pi_0\glone E \\
    \end{tikzcd} $$
\end{defn}
Note that for the rings $S$ under consideration, $\glint K(S) \simeq \glone K(S)$, so we 
alternatively wish to split the first layer of the Postnikov tower of $\glint K(S)$.  
From now on we shall consider splittings of the Postnikov tower of $\glint R$ for various commutative ring spectra $R$.
More precisely, we consider the following property for $X= \glint R$.
\begin{defn}\label{def:split_at}
We say that a connective spectrum $X$ is \emph{split at $m$} if there is an isomorphism $X \simeq X[0,m] \oplus X[m+1,\infty)$.  
\end{defn}

Thus, our main result can be now reformulated as saying that for $R= K(\Z)$ or $R = K(\F_\ell)$, the spectrum $\glint R$ is split at $1$. 
\subsection{$p$-completion}

In this section, we relate the splittings of $\glint R$ to those of $\glint R_p$ for various primes $p$.
First, we discuss the relationship between splittings of a spectrum $X$ and its $p$-completions in general.
\begin{lem} \label{maps_truncated_to_connected}
For a natural number $m$,
let $X$ be a spectrum concentrated in degrees $[0,m]$ and $Y$ be a spectrum concentrated in degrees $[m+2,\infty)$. Then, the map 
\[
[X,Y] \too \prod_p [X,Y_p]
\]
is injective.
\end{lem}

\begin{proof}
Let $F$ be the fiber of the map $Y\to \prod_p Y_p$.
By the long exact sequence associated with the fiber sequence 
\[
\Map(X,F) \too  \Map(X,Y) \too \prod_p \Map(X,Y_p) 
\]
it would suffice to show that $\pi_0 \Map(X,F) = 0$. 
By the arithmetic fracture square of $Y$, we see that $F$ is a rational spectrum, and since both $Y$ and $\prod_pY_p$ are concentrated in degrees $[m+2,\infty)$, $F$ is concentrated in degrees $[m+1,\infty)$.
Since there are no non-trivial maps between rational spectra concentrated in disjoint homotopical degrees, we deduce that 
$
[X,F] = 0$
and the result follows.
\end{proof}

By a simple instance of obstruction theory, this allows us to relate the splittability of $X$ and of its $p$-completions. 
\begin{cor}\label{p_complete_splitting}
Let $X$ be a connective spectrum with finitely generated homotopy groups. Then $X$ splits at $m$ if and only if $X_p$ splits at $m$ for all primes $p$. 
\end{cor}

\begin{proof}
First, if $X$ splits then $X\simeq X[0,m]\oplus X[m+1,\infty)$. Taking $p$ completion commutes with the truncations of $X$ since $X$ has finitely generated homotopy. Hence, by $p$-completing the splitting above we obtain a splitting 
\[
X_p \simeq X_p[0,m]\oplus X_p[m+1,\infty)
\] 
as desired. 

Conversely, note that 
$X$ splits if and only if the connecting map $\alpha\colon X[0,m] \to \Sigma X[m+1,\infty)$ of the cofiber sequence 
\[
X[m+1,\infty) \too X \too X[0,m]
\]
is null homotopic. By \Cref{maps_truncated_to_connected}, this is equivalent to the vanishing of each of the compositions 
\[
X[0,m] \xrightarrow{\alpha} \Sigma X[m+1,\infty) \too \Sigma X[m+1,\infty)_p \simeq \Sigma X_p[m+1,\infty).   
\] 
The above composition fits as the diagonal dashed arrow in the commutative square 
\[
\xymatrix{
X[0,m] \ar^\alpha [r] \ar[d] \ar@{..>}[rd] & \Sigma X[m+1,\infty) \ar[d] \\
X_p[0,m] \ar^{\alpha_p}[r] & \Sigma X_p[m+1,\infty)
}
\]
in which the vertical maps are induced from the $p$-completion $X\to X_p$. By our assumption that $X_p$ is split at $m$, the bottom horizontal map labeled $\alpha_p$ is $0$, and hence also the diagonal dashed arrow, yielding the result.  
\end{proof}

We turn to discuss splittings of $\glint R$ for a commutative ring spectrum $R$. 

\begin{lem}\label{glint_completion}
Let $R$ be a commutative ring spectrum with finitely generated homotopy groups. Then, for every prime $p$ the $p$-completion map $R\to R_p$ induces an isomorphism 
\[
(\glint R)_p \simeq (\glint R_p)_p.  
\]
\end{lem}

\begin{proof}
Note that since $\glint R$ has finitely generated homotopy groups,
$\pi_i \glint R \to \pi_i (\glint R)_p$ is a (classical) $p$-completion. In fact, since $\glint R_p$ has finitely generated $\pi_0$ and $p$-complete connected cover, the same holds for it.
It would therefore suffices to show that the map $\pi_i\glint R \to \pi_i(\glint R_p)$ induces an isomorphism on (classical) $p$-completions. For $i=0$ this map is already an isomorphism by definition (both groups are $\{\pm 1\}$), and for $i>0$ this map can be identified with the map $\pi_i R \to \pi_i R_p$, which is a $p$-completion since $R$ has finitely generated homotopy groups.    
\end{proof}

We are ready to reduce the splitting of $\glint R$ to that of the spectra $\glint R_p$. 
\begin{cor}\label{arithfracreduction}
Let $R$ be a commutative ring spectrum with finitely generated homotopy groups and let $m \ge 1$. If $\glint R_p$ splits at $m$ for every prime $p$, then $\glint R$ splits at $m$.   
\end{cor}

\begin{proof}
By the ``if'' direction of \Cref{p_complete_splitting}, it would suffice to show that $(\glint R)_p$ splits at $m$ for every prime $p$. By \Cref{glint_completion}, this spectrum is isomorphic to $(\glint R_p)_p$. By the ''only if' direction of \Cref{p_complete_splitting}, this follows from the assumption that $\glint R_p$ splits at $m$.
\end{proof}

\subsection{$T(1)$-localization}
In this section we use Lichtenbaum-Quillen type results to reduce the calculation of $\glone K(S)_p$ to that of $\glone  \LKone K(S)$ for the rings $S$ under consideration.
As usual, we denote by $\LKone X$ the localization of the spectrum $X$ at the Bousfield class of $T(1)=\S/p^k[v_1^{-1}]$ and leave the prime $p$ implicit\footnote{Note that this is the same as the $K(1)$-localization, and we use the $T(1)$-localization mostly for notation reasons.}. The canonical map $X\to\LKone X$ factors through the $p$-completion; in fact $\LKone X\to \LKone X_p$ is an equivalence. The localization functor is symmetric monoidal, so if $R$ is a commutative ring spectrum, the localization map is one of commutative ring spectra. 
We shall need the following fact, which is a special case of the Lichtenbaum-Quillen conjecture.
\begin{prop}\label{LQprop}
For $S=\Z[1/p]$ or $\F_\ell$ with $\ell\neq p$, the canonical map
\[
K(S)_p \longrightarrow \LKone K(S)
\]
identifies $K(S)_p$ with the connective cover of $\LKone K(S)$. In other
words, it induces an equivalence
\[
K(S)_p \simeq \LKone K(S)[0,\infty).
\]
\end{prop}

\begin{proof}
Suppose first that $S=\F_\ell$, with $\ell\neq p$. In this case the result
follows essentially from Quillen's calculation \cite{quillen}; indeed, $K(\F_\ell)_p$ is the connective cover of the fiber of  $\ku_p \oto{\psi_\ell - 1} \ku_p$ where $\psi_\ell$ is the $\ell$-th Adams operation. Hence, the claim follows from the fact that $\ku_p$ is the connective cover of its $T(1)$-localization $\KU_p$.

We now consider $S=\Z[1/p]$. Suppose first that $p$ is odd. As discussed in \cite[S1.4]{dwyermitch}, the Lichtenbaum--Quillen
conjecture for rings of integers with $p$ inverted asserts that the canonical
map
\[
K(S)_p \longrightarrow \LKone K(S)
\]
induces an isomorphism on $\pi_i$ for every $i\geq 1$. The
Lichtenbaum--Quillen conjecture is now a theorem, as a consequence of the
norm residue theorem of Rost--Voevodsky; see
\cite{voevod} and, e.g., the discussion in \cite[Chapter VI]{K-book}. Hence
\[
\pi_i K(\Z[1/p])_p
\xrightarrow{\;\sim\;}
\pi_i \LKone K(\Z[1/p])
\qquad (i\geq 1).
\]
It remains to consider $\pi_0$. Thomason's \'etale descent spectral
sequence \cite[Theorems~4.1 and~4.11]{thomason} (in the form for rings of integers used in
\cite[S1.4]{dwyermitch}) gives a short exact sequence
\[
0\longrightarrow
H^2_{\et}(\Z[1/p];\Z_p(1))
\longrightarrow
\pi_0\LKone K(\Z[1/p])
\longrightarrow
H^0_{\et}(\Z[1/p];\Z_p)
\longrightarrow 0.
\]
The first term vanishes. Indeed, the Kummer sequence gives
\[
0\longrightarrow
\Pic(\Z[1/p])/p^n
\longrightarrow
H^2_{\et}(\Z[1/p];\mu_{p^n})
\longrightarrow
\Br(\Z[1/p])[p^n]
\longrightarrow 0.
\]
Here $\Pic(\Z[1/p])=0$, and the exact sequence
\[
0 \to \Br(\Z[1/p]) \to \Br(\Q_p)\oplus \Br(\R) \oto{\inv_p + \inv_\infty} \Q/\Z
\]
implies that
$\Br(\Z[1/p])[p^\infty]=0$:
We have $\Br(\Q_p)\simeq \Q/\Z$ and $\Br(\R)\simeq \Z/2$, where the map $\inv_p$ is an isomorphism. Hence, $\Br(\Z[1/p])\simeq \Z/2$ and it has no $p$-primary torsion.

We deduce that
\[
H^2_{\et}(\Z[1/p];\Z_p(1))\simeq \lim_{n} H^2_{\et}(\Z[1/p];\mu_{p^n}(1))\simeq  0.
\]
Since
\[
H^0_{\et}(\Z[1/p];\Z_p)=\Z_p,
\]
we obtain
\[
\pi_0\LKone K(\Z[1/p])\simeq\Z_p.
\]
Since also $K_0(\Z[1/p])_p\simeq\Z_p$, and the comparison map between them is easily seen to be the identity on $\pi_0$, we deduce that
\[
K(\Z[1/p])_p
\simeq
\LKone K(\Z[1/p])[0,\infty)
\]
for odd $p$.

For $p=2$, the result follows instead from the homotopy Cartesian square (see \cite[Theorem~70]{mitchellK1})
\[
\begin{tikzcd}
K(\Z[1/2])_2 \arrow[r] \arrow[d]
    & K(\F_3)_2 \arrow[d] \\
\ko_2 \arrow[r]
    & \ku_2 .
\end{tikzcd}
\]
The upper right, lower right, and lower left corners are each the connective cover of their $T(1)$-localizations, and hence the upper left corner is as well. 
\end{proof}
From this, we can relate the spectrum of units of $K(S)$ and its $T(1)$-localization.

\begin{lem}\label{K1reduction}
For $R= K(\F_\ell)$ or $R= K(\Z)$, we have the following:
\begin{enumerate}
\item For a prime $\ell\ne p$: 
\[
\glone K(\F_\ell)_p \simeq \glone L_{T(1)} K(\F_\ell).
\]
\item There is a cofiber sequence
\[
\Z_p\rightarrow \glone  K(\Z)_p\rightarrow \glone\LKone K(\Z).
\]
\end{enumerate}

\end{lem} 
\begin{proof}
For $(1)$, this follows from \Cref{LQprop} and the fact that $\glone$ depends only on the connective cover of a commutative ring spectrum. 

For $(2)$, again by \Cref{LQprop} we can identify the map 
$K(\Z)_p \too \LKone K(\Z)$ on connective covers with the map 
\[
K(\Z)_p \too K(\Z[1/p])_p. 
\]
Taking $p$-completion from the localization sequence 
\[
K(\Z)\rightarrow K(\Z[1/p]) \too \Sigma K(\F_p)
\]

and using that by Quillen's computation $K(\F_p)_p \simeq \Z_p$ we get the map 
$
K(\Z)_p \to K(\Z[1/p])_p
$
induces an isomorphism on $\pi_i$ for $i\ne 1$ and an injection on $\pi_1$ with cokernel $\Z_p$. Hence, the same holds for the map $\glone K(\Z)_p \to \glone K(\Z[1/p])_p$ and the result follows. 
\end{proof}

\subsection{The homotopy type of $\glint R[0,1]$}\label{glintsection}
We shall now determine the homotopy type of the $1$-truncation of $\glone R$ for the ($p$-complete) ring spectra $R$ under consideration. This calculation will be useful later in the construction of sections for the map $\glone R \to \glone R[0,1]$.

\begin{rem} \label{rem:eta_1_truncated}
Recall that a connective, 1-truncated spectrum $X$ is determined by the data of $\pi_0 X, \pi_1 X$ and the multiplication by the Hopf element $\eta\in \pi_1\Sph$, i.e. the map 
\[
\eta\cdot \colon \pi_0 X / 2 \too \pi_1X
\]
More precisely, one recovers $X$ from this data as the fiber of the composition 
\[
\pi_0 X \onto \pi_0 X /2 \oto{\Sq^2} \Sigma^2 \pi_0X/2 \oto{\eta} \Sigma^2 \pi_1X. 
\]
In particular, $X$ splits as a sum $X\simeq \pi_0 X \oplus \Sigma \pi_1 X$ if and only if $\eta$ acts by $0$ on $\pi_0 X$.  
\end{rem}

Returning to the determination of $\glint R[0,1]$, first we note that the case of an odd prime $p$ is easy. 
\begin{lem}\label{gltrunc_odd_p}
Let $R$ be a $p$-local commutative ring spectrum for an odd prime $p$. Then 
\[
\glint R[0,1] \simeq \Z/2 \oplus \Sigma \pi_1 R.   
\]
\end{lem} 

\begin{proof} 
By definition $\pi_0 \glint R \simeq \Z/2$ and $\pi_1 \glint R \simeq \pi_1 \glone R \simeq \pi_1 R$. Since $\pi_1 R$ is $p$-local and hence has no $2$-torsion, the map $\eta\colon \pi_0 \glint R \to \pi_1 \glint R$ is the zero map and we obtain the desired splitting by \Cref{rem:eta_1_truncated}.
\end{proof}

We turn to the more interesting case of $p=2$.  

\begin{lem}\label{lemmasq2glint}
   For a prime $\ell \ne 2$, the spectrum $\glint (K(\F_\ell)_2)[0,1] \simeq \glint \LKone K(\F_\ell)[0,1]$ is the fiber of the map $$\Z/2\xrightarrow{Sq^2}\Sigma^2\Z/2\xrightarrow{1\mapsto -1}\Sigma^2(\F_\ell)^\times_2.$$
\end{lem}
\begin{proof}
Since $\pi_1 K(\F_\ell) \simeq \F_\ell^\times$ and $\pi_0 K(\F_\ell) \simeq \Z$, we see that $\glint K(\F_\ell)_2$ has the correct homotopy groups. By \Cref{rem:eta_1_truncated}, it remains to show that the map 
\[
\eta\colon \pi_0 \glint (K(\F_\ell)_2) \to \pi_1 \glint (K(\F_\ell)_2)
\]
sends the non-trivial element of the group $\pi_0\glint (K(\F_\ell)_2) \simeq \Z/2$ to the element $-1 \in \pi_1 \glint K(\F_\ell)_2 \simeq (\F_\ell^\times)_2$. 

Now, there is a canonical map of spectra 
\[
\pic(\F_\ell) \to \glint K(\F_\ell) \to \glint K(\F_\ell)_2
\]
which is given on $\pi_0$ and $\pi_1$ by the reduction mod $2$ and $2$-completion maps respectively:
\[
\Z \onto \Z/2 \quad \text{and} \quad \F_\ell^\times \onto (\F_\ell^\times)_2. 
\]
Hence, it would suffice to show that the map $\eta\colon \pi_0 \pic(\F_\ell) \to \pi_1\pic(\F_\ell)$ sends the generator $\Sigma \F_\ell\in \pic(\F_\ell)$ to the element $-1\in \pi_1 \pic(\F_\ell) \simeq \F_\ell^\times$. This follows from the facts that on the Picard spectrum of a ring $S$ we have 
$
\eta\cdot L = \dim(L) 
$
(see, e.g., \cite[Proposition 3.20]{Telecyclo}) and that $\dim(\Sigma S) = -1 \in S^\times$. 


\end{proof}

We turn to the case $R= K(\Z)$. 

\begin{lem}\label{glintKZchunk}
    The spectrum $\glint K(\Z)[0,1]\simeq \glint K(\Z)_2[0,1]$ is the fiber of $Sq^2:\Z/2\rightarrow\Sigma^2\Z/2$.
\end{lem}
\begin{proof}
   Since $\pi_1 K(\Z) \simeq \Z/2$, by \Cref{rem:eta_1_truncated} it suffices to show that the map $\eta \colon \pi_0 \glint K(\Z) \to \pi_1 \glint K(\Z)$ is the non-zero map. Here, as in the case of $\F_\ell$, one can use the map $\pic(\Z) \to \glint K(\Z)$. Indeed, this map is given on $\pi_0$ and $\pi_1$ by the maps $\Z\onto \Z/2$ and $\Z/2 \simeq \Z/2$ respectively, so it is enough to show that $\eta$ takes the generator of $\pi_0 \pic(\Z)$ to a non-zero class. This follows as above from the fact that 
    \[
    \eta\cdot \Sigma \Z = -1 \in \pi_1\pic(\Z).
    \]
\end{proof}

Passing from the $2$-completion to the $T(1)$-localization introduces an extra summand to $\glint$.

\begin{lem}\label{glintLT1KZchunk}
At the prime $p = 2$, we have 
    \[
    \glint \LKone K(\Z)[0,1]\simeq \Sigma\Z_2\oplus \glint K(\Z)[0,1]
    \]
\end{lem}
\begin{proof}
    Recall that $\LKone K(\Z)[0,\infty) \simeq K(\Z[1/2])_2$. 
    It follows from \Cref{K1reduction} that we have a cofiber sequence 
    \[
    \glint K(\Z) \too \glint \LKone K(\Z) \too \Sigma \Z_2. 
    \]
    Since the map $\pi_1\LKone K(\Z) \too \pi_1\Sigma \Z_2 = \Z_2$ can be identified with  the $2$-completion of the $2$-adic valuation map $\Z[1/2]^\times \to \Z$ it is in particular a surjection. Hence, the above sequence induces a cofiber sequence
    \[
    \glint K(\Z)[0,1] \too \glint \LKone K(\Z)[0,1] \too \Sigma \Z_2
    \]
    between the $1$-truncations. 
    The resulting boundary map $\Sigma \Z_2 \to \Sigma \glint K(\Z)[0,1]$ is trivial on $\pi_1$ and hence lifts to a map of spectra 
    \[
    \Sigma \Z_2 \too \Sigma^2 \pi_1 \glint K(\Z) \simeq \Sigma^2\Z/2.  
    \]
    Since $\pi_0\Map_{\Sp}(\Sigma \Z_2,\Sigma ^2 \Z/2) = 0$ the boundary map must vanish and the sequence splits. 
    
\end{proof}

\subsection{Splitting $\slone R$}
\label{sl1split}

One can further simplify $\glone R$ by removing its $\pi_0$ completely. The resulting spectrum is $\slone R := \glone R[1,\infty)$. Splitting the bottom homotopy group from $\slone R$ is strictly easier in general than splitting $\glint R[0,1]$ from $\glint R$. Indeed, a splitting 
\[
\glint R \simeq \glint R[0,1] \oplus \glint R [2,\infty) 
\]
induces, by taking connected covers, a splitting 
\[
\slone R \simeq \Sigma \pi_1 \slone R \oplus \glint R[2,\infty).
\]
While the converse does not hold in general, splitting $\slone R$ is an important step towards splitting $\glint R$, so we show the existence of such splittings for our ring spectra of interest first.

\begin{lem}
    Let $R$ be a commutative ring spectrum such that the truncation map  $\slone R\rightarrow \slone R[0,1]\simeq \Sigma\pi_1\glone R$ is split. Then $\pi_1\glone R$ is a subgroup of $\pi_1\G_mR$ via a map $\pi_1\glone R=[\Sigma\Z,\Sigma \pi_1\glone R]\rightarrow [\Sigma\Z, \slone R].$

\end{lem}

\begin{proof}
    This is immediate. The desired map is induced by any choice of splitting $\Sigma\pi_1\glone R=\slone R[0,1]\rightarrow \slone R$.
\end{proof}

\begin{lem}\label{splitting_slone}
    Let $S$ be a commutative ring for which $\pi_1K(S) \simeq S^\times $ (e.g., a local ring). Then, there is a splitting of $\slone K(S)$ as 
$$
\slone K(S)\simeq \Sigma(S^\times) \oplus \slone K(S)[2,\infty). 
$$
If, moreover, $K(S)$ has finitely generated homotopy groups, we get a corresponding factorization 
$$
\slone K(S)_p\simeq \Sigma(S^\times)_p \oplus \slone K(S)_p[2,\infty). 
$$
\end{lem} 
\begin{proof}
     The result after $p$-completion follows immediately from the first claim by taking $p$-completions from the decomposition. 
     The map $\pic(S)\rightarrow \glone K(S)$ restricts to a map 
     \[
     \Sigma S^\times \simeq \pic(S)[1,\infty) \rightarrow \glone K(S) [1,\infty) \simeq \slone K(S)  
     \]
     which is a section of the $1$-truncation map and hence induces the desired splitting.
\end{proof}

\section{Strict units and the logarithmic fiber} 
\subsection{Preliminaries}
The main advantages of a $T(1)$-local ring spectrum in studying its units are the logarithm map and the power operation $\theta$. 
The first is a morphism of spectra 
\[
\log \colon \glone R \too R, 
\]
defined by Rezk using the Bousfield-Kuhn functor (\cite{rezk}). 
\begin{defn}
For a $T(1)$-local ring spectrum $R$, we let $\fiblog{R}$ be the fiber of the map $\log\colon \glone R \to R$.
\end{defn}




Closely related to the logarithm is the power operation $\theta$, which defines a map $\theta\colon \Omega^\infty R \to \Omega^\infty R$. When restricted to the units, it precisely cuts out the space of strict units via the fiber sequence (of spaces!) 
\[
\Omega^\infty \G_m(R) \rightarrow \Omega^\infty \glone R \oto{\theta} \Omega^\infty R. 
\]
Observe that the basepoint of the middle space is the multiplicative unit $1\in\pi_0R$, while that of the right space is the additive unit $0\in\pi_0R$, and that $\theta(1)=0$. 

The main calculations of this section are based on a comparison between the long exact sequences associated to $\log$ and $\theta$. We will use a few key lemmas.

\begin{lem}\label{Zloglemma}
    Let $R$ be a $T(1)$-local ring. Then 
    $$\hom(\Z,\fiblog R)\simeq \G_m R.$$
\end{lem}
\begin{proof}
    Since $R$ is $T(1)$-local, we have $\hom(\Z,R) = 0$. Hence, 
    the result follows by applying $\hom(\Z,-)$ to the fiber sequence 
    \[
    \fiblog R \too \glone R \oto{\log} R.
    \]
\end{proof}

\begin{lem}\label{logLES}
    Let $R$ be a $T(1)$-local ring. Then 
    $\fiblog(R)$ is 2-truncated and $\pi_2\fiblog(R)$ is torsion free.
\end{lem}
\begin{proof}
    For an $L_1$-local commutative ring spectrum $R$, let $\fibahr{R}$ be the fiber of the $L_1$-localization map $\glone R \to L_1 \glone R$. By \cite[Theorem 4.11]{ahr}, $\fibahr(R)$ is $1$-truncated. Consider the commutative diagram where the rows are fiber sequences
    $$\begin{tikzcd}
        \fibahr R\arrow[r]\arrow[d]& \glone R\arrow[r]\arrow[d]& L_1\glone R\arrow[d]\\
        \fiblog R \arrow[r]& \glone R\arrow[r]& \LKone \glone R\simeq R\\
    \end{tikzcd}.$$
    The right vertical arrow is the top row of the chromatic fracture square, so its fiber agrees with that of the bottom row of the chromatic fracture square which is a rational spectrum. Call this fiber $Q$. Note also that the fiber of the map $\fibahr R \to \fiblog R$ agrees with $\Omega Q$, so that we have a cofiber sequence
    \[
    \fibahr R \to \fiblog R \to Q.
    \]
    Writing the long exact sequence of homotopy groups associated with this cofiber sequence and using that $\fibahr R$ is $1$-truncated we get 
    \[
    0 \to \pi_2\fiblog R \into \pi_2 Q \to \pi_1 \fibahr R. 
    \]
    and 
    \[
    \pi_i\fiblog R \iso \pi_i Q \quad \text{for}\quad i\ge 3 
    \]

    From the first sequence we see that $\pi_2\fiblog R$ embeds in a rational abelian group and hence it is torsion free. 
    From the second displayed isomorphisms we deduce that $\pi_i\fiblog R$ is a rational vector space for $i\ge 3$. However, already for $i\ge 1$ the group $\pi_i \fiblog R$ agrees with the $i$-th homotopy group of the fiber of the map $\glone R[1,\infty] \to R[1,\infty]$. This is the fiber of a map between $p$-complete spectra, hence its homotopy groups are derived $p$-complete. Since an abelian group which is both $p$-complete and rational is zero, the groups $\pi_i \fiblog R$ vanish for $i\ge 3$ and $\fiblog R$ is 2-truncated.
\end{proof}

\begin{rem}
    It is possible for $\fiblog R$ to have nontrivial $\pi_2$. 
    Indeed, exactly that happens in the case of $R=KU_p$.

\end{rem}

\begin{lem}\label{ahrz}
    For any $T(1)$-local ring $R$, $\G_m R$ is $2$-truncated and $\pi_2 \G_m R$ is torsion free.
\end{lem}
\begin{proof}

    First, $\Omega \G_m R  \simeq \hom(\Z, \Omega \glone R)$ is $p$-complete since $\Omega \glone R$ is (where the loops are computed in \emph{connective spectra}). Let $\mu_p(R):= \hom(\Z/p,\glone R)$. Hence, by the long exact sequence of homotopy groups  associated with the fiber sequence 
    \[
    \mu_p R \too \G_m R \oto{p} \G_m R, 
    \]
    the claim is equivalent to the fact that $\mu_p(R)$ is $1$-truncated, which follows from the vanishing of the $T(1)$-local suspension spectrum of $B^2\Z/p$ (see, e.g., \cite[Proposition 4.1]{Telecyclo}). 
    
    We can also argue using the logarithmic fiber: the truncatedness claim follows immediately from  \Cref{Zloglemma} and \Cref{logLES}, while the claim about $\pi_2$ now follows from the second part of \Cref{logLES} and the fact that, since $\fiblog R$ is 2-truncated, 
    \[
    \pi_2\G_m R\simeq \pi_2\hom(\Z,\fiblog R) \simeq \pi_2\fiblog R.
    \]
\end{proof}

\begin{lem}\label{keylogthetaformula}
    Let $R$ be a $T(1)$-local ring. Suppose that $\theta_1\colon \pi_1 R \to \pi_1 R$ (at the basepoint $1\in R$) is zero. Then the map  $\log_1 \colon \pi_1 R \to \pi_1R$  is also zero.
\end{lem}
\begin{proof}
    The general formula for the relationship between log and $\theta$ in degree 0 (i.e. between $\log_0$ and $\theta_0$) is Theorem 1.9 of \cite{rezk}. In loc. cit. just below, there is a simplified formula: if  $\epsilon^2=0$ then $\log_0(1+\epsilon)=\epsilon-\theta_0(\epsilon)$. 

    Let $x$ be an element in $\pi_1R$. Write $\alpha$ for the square zero suspension class in $R^1(S^1)$. Note that $R^*(S^1)$ is a free $\pi_*R$-module on $1$ and $\alpha$. For $z=z_11+z_\alpha\alpha\in R^*(S^1)$ write $\langle \alpha ,z\rangle:=z_\alpha$ and $\langle 1 ,z\rangle:=z_1$ for the coefficients in $\pi_*R$. Then 
    $$\theta_1(x)=\langle \alpha,\theta_0(1+x\alpha)\rangle,$$ $$\log_1(x)=\langle \alpha,\log_0(1+x\alpha)\rangle.$$
    Since $\alpha^2=0$ we may apply Rezk's simplified logarithm formula to deduce that 
    $$\log_1(x)=\langle \alpha,\log_0(1+x\alpha)\rangle=\langle \alpha, x\alpha-\theta_0(x\alpha)\rangle.$$
    On the other hand $\theta_0$ is a $\delta$-structure\footnote{It satisfies the equations required to make $x\mapsto x^p+p\theta_0(x)$ a ring map.} (\cite{rezk} 1.8) so that $\theta_0(1+\epsilon)=\theta_0(\epsilon)-\epsilon$ whenever $\epsilon^2=0$. Hence
    $$\langle \alpha, x\alpha-\theta_0(x\alpha)\rangle=\langle \alpha, -\theta_0(1+x\alpha)\rangle.$$
    Stringing these together gives
    $$\log_1(x)=-\theta_1(x).$$
\end{proof}

\subsection{Main results}

The main results of this section are the calculations of strict units: \Cref{pistarGmKFlK1loc}, \Cref{pistarGmKZK1locoddp}, \Cref{pistarGmKZK1locp2} and the homotopy groups of the logarithmic fiber: \Cref{logfibhomoKFl}, \Cref{logfibhomoKZoddp}, \Cref{logfiphomoKZp2}. 

We summarize them in the following table for the reader's convenience. The symbols $\ltimes$ and $A$ appearing in the third column are extension problems that will be resolved in a later section (cf. \Cref{logfiberKFlp2} for the former and \Cref{logfiphomoKZp2}, \Cref{remA} for the latter).

\begin{tabular}{c|c|c}
 \text{Ring} R    & $\G_m\LKone R[0,\infty)$ & $\pi_{0,1}\fiblog R$                          \\ \hline 
   $K(\F_\ell), \quad \ell \ne p$ odd                 & $\F_p^\times \oplus (\F_\ell^\times)_p \oplus \Sigma (\F_\ell^\times)_p$  & $\F_p^\times \times (\F_\ell^\times)_p,\ (\F_\ell^\times)_p$    \\  
   $K(\F_\ell), \quad \ell \ne p=2$                  & $(\F_\ell^\times)_2 \oplus \Sigma (\F_\ell^\times)_2$  & $\Z/2\ltimes (\F_\ell^\times)_2,\ (\F_\ell^\times)_2$    \\
   $K(\Z), \quad p>2$                  & $\Z_p \oplus \F_p^\times\oplus  \Sigma \Z_p$  & $\Z_p\times \F_p^\times, \Z_p$    \\
   $K(\Z), \quad p=2$                  & $\Z_2 \oplus\Z/2\oplus  \Sigma \Z_2\oplus\Sigma\Z/2$  &  $\Z_2\times A$,\ $\Z_2\times\Z/2$   \\
\end{tabular}

\subsection{Strict units in $\LKone K(\F_\ell)$, $p\neq \ell$}

In this section we assume our two primes $p$ and $\ell$ are not equal\footnote{If they are, note that $\LKone K(\F_\ell)$ is zero.}.

The key facts about $R:=\LKone K(\F_\ell)$ are
\begin{enumerate}
    \item the splitting of $\slone R$ (cf. \Cref{sl1split}), and
    \item $\pi_{0,1,2}R$ are 
$\Z_p, (\F_\ell^\times)_p, 0$ and $\pi_3R$ is torsion (cf. \cite{quillen}).
\end{enumerate}
Now we simply plug the key facts into the computational tools: the $\theta$ LES becomes 
$$
\begin{tikzcd}[row sep=tiny]
    \pi_3\G_mR\arrow[r]&\pi_3R\arrow[r,"\theta_3"]&\pi_3R\\
    \pi_2\G_mR\arrow[r]&0\arrow[r,"\theta_2"]&0\\
\pi_1\G_mR\arrow[r]&(\F_\ell^\times)_p\arrow[r,"\theta_1"] &(\F_\ell^\times)_p\\
\pi_0\G_mR\arrow[r]&\Z_p^\times\arrow[r,"\theta_0"]& \Z_p\\
\end{tikzcd}.
$$
We see immediately that $\pi_1\G_mR$ is a subgroup of $(\F_\ell^\times)_p$. But from the splitting of $\slone R$ (cf. \Cref{splitting_slone}), $\pi_1\G_mR$ contains $\pi_1\glone R$ as a subgroup and must therefore be equal to it. That forces $\theta_1$ to be zero. Note that\footnote{In degree 0, $\theta$ is the unique delta structure on $\Z_p$, whose kernel (among units) are the units $\F_p^\times$} the kernel (i.e. preimage of $0$) of the function $\theta_0$ is $\F_p^\times$. We record the consequences as a lemma for reference purposes.
\begin{lem}\label{pistarGmKFlK1loc}
    Let $R=\LKone K(\F_\ell)$. For all $p\neq \ell$, $\pi_1\G_mR\simeq (\F_\ell^\times)_p$ and there is a (uniquely) split short exact sequence 
    $$0\rightarrow (\F_\ell^\times)_p\rightarrow\pi_0\G_mR\rightarrow \F_p^\times\rightarrow 0.$$

\end{lem}
\begin{proof}
    The statement about $\pi_1$ and the short exact sequence have been established in the discussion leading up to the lemma. It is uniquely split since for odd $p$ the left and right groups are torsion groups of coprime order and for $p=2$ the right group is trivial.
\end{proof}

Now we turn to the logarithmic fiber, again with $R:= L_{T(1)} K(\F_\ell)$. Since $\fiblog R$ is $2$-truncated and $\hom(\Z,\fiblog R) \simeq \G_m R$ (\Cref{Zloglemma}), the top non-vanishing homotopy group of $\fiblog R$ agrees with that of $\G_m R$. Hence, 
\[
\pi_2 \fiblog R = 0 \quad  \text{and} \quad  \pi_1\fiblog R\simeq\pi_1\G_mR=(\F_\ell^\times)_p.
\]
Now consider the long exact sequence associated with the log fiber sequence:

$$
\begin{tikzcd}[row sep=tiny]
    0\arrow[r]&0\arrow[r,"\log_2"]&0\\
\pi_1\fiblog R=(\F_\ell^\times)_p\arrow[r]&(\F_\ell^\times)_p\arrow[r,"\log_1"] &(\F_\ell^\times)_p\\
\pi_0\fiblog R\arrow[r]&\Z_p^\times\arrow[r,"\log_0"]& \Z_p\\
\end{tikzcd},
$$
Since $(\F_\ell^\times)_p$ is a finite group, we find that $\log_1=0$. Now $\log_0$ is the pre-composition of the $p$-adic logarithm with the $x^{p-1}$ map, so its kernel is tor$(\Z_p^\times)$. Again, we collect the consequences in a lemma.
\begin{lem}\label{logfibhomoKFl}
    For all $p\neq \ell$ and $R := \LKone K(\F_\ell)$, we have that $\pi_2\fiblog R=0$, $\pi_1\fiblog R\simeq (\F_\ell^\times)_p$, and there is a short exact sequence
$$0\rightarrow (\F_\ell^\times)_p\rightarrow\pi_0\fiblog R\rightarrow (\Z_p^\times)_\tors\rightarrow 0.$$
When $p$ is odd the sequence is split.
\end{lem}
The fact that it also splits at $p=2\neq\ell$ will follow from the integral case by functoriality (cf. \Cref{remA}).

\subsection{Strict units in $\LKone K(\Z)$}\label{subsec:strict_units_K(1)_local}
We turn to calculate the strict units and logarithmic fibers of the $T(1)$-local $K$-theory of $\Z$. 
\subsubsection{The case $p>2$}\label{section_LT1K(Z)p_odd}
The key facts about $R:= \LKone K(\Z)$ are
\begin{enumerate}
    \item The identification $\LKone K(\Z)[0,\infty) \simeq K(\Z[1/p])_p$. 
    \item The corresponding splitting of $\slone R$ (cf. \Cref{sl1split}),
    \item  $\pi_0R=\Z_p$, $\pi_1R=\Z_p$, and $\pi_{2}R=0$, and $\pi_3R$ is torsion (in fact it's zero for $p>3$)(cf. \cite{blumman1}, bottom of page 7 and top half of page 10. See also \cite{blumman2} middle of page 4).
\end{enumerate}
The $\theta$ LES becomes
$$
\begin{tikzcd}[row sep=tiny]
    \pi_3\G_mR\arrow[r]&\pi_3R\arrow[r,"\theta_3"]&\pi_3R\\
    \pi_2\G_mR\arrow[r]&0\arrow[r,"\theta_2"]&0\\
\pi_1\G_mR\arrow[r]&\Z_p\arrow[r,"\theta_1"] & \Z_p\\
\pi_0\G_mR\arrow[r]&\Z_p^\times\arrow[r,"\theta_0"]& \Z_p\\
\end{tikzcd}.
$$
Since $\pi_3 R$ is torsion, so is $\pi_2\G_mR$. By \Cref{ahrz} we deduce that $\pi_2 \G_m R = 0$. From the splitting of $\slone R$ (cf. \Cref{splitting_slone}) we find that $\pi_1\G_mR$ contains $\pi_1\glone R \simeq \Z_p$ as a subgroup. Since it is also isomorphic to the kernel of $\theta_1$, it follows that $\theta_1 = 0$ and $\pi_1\G_mR \simeq \Z_p$. Since again the kernel of $\theta_0$ is $\F_p^\times$, we deduce that 
$$
\pi_0\G_mR=\Z_p\times \F_p^\times. 
$$ 
As usual we collect these results in a lemma.

\begin{lem}\label{pistarGmKZK1locoddp}
    Let $R=\LKone K(\Z)$ and $p>2$. Then $\pi_1\G_m R\simeq \Z_p$ and $\pi_0\G_mR\simeq\Z_p\times \F_p^\times $.
\end{lem}




Now we turn to the logarithmic fiber. Again, \Cref{Zloglemma} immediately tells us that $\pi_2\fiblog R=0$ and $\pi_1\fiblog R=\pi_1\G_mR=\Z_p$.
Now consider the log LES (cf. \Cref{logLES})
$$
\begin{tikzcd}[row sep=tiny]
    \pi_3\fiblog R=0\arrow[r]&\pi_3R\arrow[r,"\log_3"]&\pi_3R\\
    \pi_2\fiblog R=0\arrow[r]&0\arrow[r,"\log_2"]&0\\
\pi_1\fiblog R=\Z_p\arrow[r]&\Z_p\arrow[r,"\log_1"] & \Z_p\\
\pi_0\fiblog R\arrow[r]&\Z_p^\times\arrow[r,"\log_0"]& \Z_p\\
\end{tikzcd}.
$$
We find that $\log_1=0$ by the same argument as for $\theta_1=0$ above. So $\pi_0\fiblog R$ is an extension of the kernel, $\F_p^\times$, of the degree zero logarithm  by $\Z_p$. So we immediately get the following lemma.

\begin{lem}\label{logfibhomoKZoddp}
    Let $R=\LKone K(\Z)$ and $p>2$. Then 
    $$\pi_2\fiblog R=0,\ \pi_1\fiblog R\simeq \Z_p\  \text{and}\  \pi_0\fiblog R\simeq \Z_p\times \F_p^\times.$$
\end{lem}

\subsubsection{The case $p=2$}\label{calcp2}
The key facts about $R:= \LKone K(\Z)$ are
\begin{enumerate}
    \item the splitting of $\slone R$ (cf. \Cref{sl1split}),
    \item  $\pi_0R=\Z_2$, $\pi_1R=\Z_2\times \Z/2$, and $\pi_{2,3}R$ are finite \footnote{This follows from eg. the presentation of $R$ as the $T(1)$-localization of the pullback $K(\F_3)\rightarrow \mathrm{ku}\leftarrow \mathrm{ko}$.}
\end{enumerate}

The $\theta$ LES becomes
$$
\begin{tikzcd}[row sep=tiny]
    \pi_3\G_mR = 0\arrow[r]&\pi_3R\arrow[r,"\theta_3"]&\pi_3R\\
    \pi_2\G_mR\arrow[r]&\pi_2R\arrow[r,"\theta_2"]&\pi_2R\\
\pi_1\G_mR\arrow[r]&\Z_2\times\Z/2\arrow[r,"\theta_1"] & \Z_2\times\Z/2\\
\pi_0\G_mR\arrow[r]&\Z_2^\times\arrow[r,"\theta_0"]& \Z_2\\
\end{tikzcd}.
$$
As usual, this forces  $\pi_2\G_mR$ to be torsion, and hence zero by \Cref{ahrz}. So $\theta_2$ is an isomorphism, as it is an injection of finite groups. From the splitting of $\slone R$ (cf.  \Cref{splitting_slone}) we find that $\pi_1\G_mR\to\pi_1\glone R$ is an isomorphism, so that $\theta_1 = 0$ and  
$\pi_0\G_mR=\Z_2\times \Z/2 $
(note that at $p=2$, $\theta_0$ is injective).  As usual we collect some facts in a lemma for reference.
\begin{lem}\label{pistarGmKZK1locp2}
    Let $R:= \LKone K(\Z)$ and $p=2$. Then 
    $$\pi_1\G_mR\simeq \pi_0\G_mR\simeq \Z_2\times\Z/2.$$
\end{lem}



Now we turn to the logarithmic fiber $\fiblog R$. Again, \Cref{Zloglemma} immediately implies that $\pi_2\fiblog R=0$ and $\pi_1\fiblog R=\pi_1\G_mR=\Z_2\times \Z/2$.
Now consider the log LES:
$$
\begin{tikzcd}[row sep=tiny]
    \pi_3\fiblog R=0\arrow[r]&\pi_3R\arrow[r,"\log_3"]&\pi_3R\\
    \pi_2\fiblog R=0\arrow[r]&\pi_2R\arrow[r,"\log_2"]&\pi_2R\\
\pi_1\fiblog R=\Z_2\times \Z/2\arrow[r]&\Z_2\times\Z/2\arrow[r,"\log_1"] & \Z_2\times\Z/2\\
\pi_0\fiblog R\arrow[r]&\Z_2^\times\arrow[r,"\log_0"]& \Z_2\\
\end{tikzcd}.
$$
Since $\theta_1=0$, \Cref{keylogthetaformula} forces $\log_1=0$, and so $\pi_0\fiblog R$ is an extension of $\Z/2$ (the kernel of the degree 0 logarithm) by $\Z_2\times \Z/2$. We deduce the following:
\begin{lem}\label{logfiphomoKZp2}
    Let $R:= \LKone K(\Z)$ and $p=2$. Then $$\pi_2\fiblog R=0,\ \pi_1\fiblog R=\Z_2\times \Z/2,\  \text{and}\ \pi_0
\fiblog R=\Z_2\times A,$$
    where $A$ is defined to be the torsion subgroup of $\pi_0
\fiblog R$ and is either $\Z/2$, $\Z/2^2,$ or $\Z/4$. 
\end{lem}
\begin{rem}\label{remA}
    We will eventually resolve the extension problem and show that $A\simeq \Z/2^2$ (cf. \Cref{resolveAext}). 
\end{rem}


\section{The connective homotopy type of the logarithmic fiber}
The main results of this section are calculations of the connective cover of the logarithmic fibers: \Cref{logfiberKFloddp}, \Cref{logfiberKFlp2}, \Cref{logfiberKZoddp}, \Cref{logfiberKZp2}, as well as splittings of the integral units of $\LKone K(\Z)$: \Cref{glintsplitKZpodd}, \Cref{glintsplitKZp2}, and the splitting of the integral units of $K(\Z)_p$: \Cref{glintsplittingpcomplete}.

We summarize them in the following table for the reader's convenience. Write 
$$\Z/2 \ltimes_{\Sq^2} \Sigma (\F_\ell^\times)_2:=\fib (\Z/2\xrightarrow{Sq^2}\Sigma^2\Z/2\xrightarrow{1\mapsto -1}\Sigma^2(\F_\ell)^\times_2),$$
$$\Z/2 \ltimes_{\Sq^2} \Sigma\Z/2:=\fib(\Z/2\xrightarrow{Sq^2}\Sigma^2\Z/2).$$

\begin{tabular}{c|c|c}  \label{tab:log_fib}
 \text{Ring} R    & $\fiblog R[0,\infty)$ &    $\glint R[0,1]$                       \\ \hline 
   $K(\F_\ell), \quad  p$ odd                 & $\F_p^\times \oplus (\F_\ell^\times)_p\oplus\Sigma (\F_\ell^\times)_p$  & $\Z/2\oplus \Sigma(\F_\ell^\times)_p$    \\  
   $K(\F_\ell), \quad  p=2$                  &$(\F_\ell^\times)_2\oplus \glint R[0,1]$ & $\Z/2 \ltimes_{\Sq^2} \Sigma (\F_\ell^\times)_2$     \\
   $K(\Z), \quad p>2$                  & $\Z_p\oplus \F_p^\times\oplus\Sigma \Z_p$  &   $\Z/2 \oplus  \Sigma \Z_p$   \\
   $K(\Z), \quad p=2$                  & $\Z_2\oplus\Z/2\oplus \Z/2 \ltimes_{\Sq^2} \Sigma\Z/2\oplus\Sigma\Z_2$  & $\Z/2\ltimes_{\Sq^2}\Sigma\Z/2 \oplus \Sigma \Z_2$    \\
\end{tabular}

\subsection{For $R=K(\F_\ell)_p$}

When $p\neq\ell$ we have already calculated the homotopy groups of the logarithmic fiber (\Cref{logfibhomoKFl}). For odd $p$ we can immediately describe the connective homotopy type since there is no possibility of a $k$-invariant. 

\begin{lem}\label{logfiberKFloddp}
Let $R=K(\F_\ell)_p$ with $\ell\neq p>2$. Then the connective cover of the logarithmic fiber is 
    $$\fiblog R [0,\infty)\simeq (\F_\ell^\times)_p\oplus \F_p^\times\oplus\Sigma (\F_\ell^\times)_p.$$
\end{lem}
\begin{proof}
The homotopy groups of $\fiblog R[0,\infty)$ are computed in \Cref{logfibhomoKFl}. Since this spectrum is $1$-truncated and $p$-local for an odd prime $p$, it splits to a direct sum according to its homotopy groups.  
\end{proof}
We also immediately get a splitting of the integral units, and this time we can include the case $\ell=p$.

\begin{lem}\label{glintsplitKFlpodd}
    Let $R=K(\F_\ell)_p$ with $ p>2$. Then
    $$\glint R\simeq \Z/2\oplus \Sigma(\F_\ell^\times)_p\oplus \glint R[2,\infty).$$
\end{lem}
\begin{proof}
    By \Cref{gltrunc_odd_p} we have $\glint R[0,1]\simeq \Z/2 \oplus \Sigma (\F_\ell^\times)_p$. Hence, it remains to show that the $k$-invariant 
    \[
    \Z/2 \oplus \Sigma (\F_\ell^\times)_p \to \Sigma \glint R [2,\infty)
    \]
    is trivial. The $\Sigma (\F_\ell^\times)_p$-component of this map vanishes by the splitting of $\slone$ (\Cref{sl1split}), and the $\Z/2$-component vanishes since $\Sigma \glint R[2,\infty)$ is $p$-local and hence receives no maps from $\Z/2$.

\end{proof}

When $p=2$ we had an unresolved extension problem in $\pi_0\fiblog R=(\F_\ell^\times)_2\rtimes\Z/2$ (where here again the symbol $\rtimes$ stands for an unspecified extension of the two groups). Note that because $\pi_0\G_mR=(\F_\ell^\times)_2$ is half the size of $\pi_0\fiblog R$ (cf. \Cref{pistarGmKFlK1loc}), there must be a $k$-invariant connecting $\pi_0$ and $\pi_1$ of $\fiblog R$. In other words, multiplication by $\eta$ must be non-trivial on the homotopy of this spectrum. But we can say more.

\begin{lem}\label{splitiffFl}
    Let $R=K(\F_\ell)_2$ with $\ell> 2$. Then $\glint R$ splits at $1$ if and only if the surjection of groups $\pi_0\fiblog R=(\F_\ell^\times)_2\rtimes\Z/2\rightarrow \Z/2$ is split. 
\end{lem}
\begin{proof}
Since $R$ is the connective cover of a $T(1)$-local ring, we make a small abuse of notation and write $\fiblog R$ for $\fiblog L_{T(1)}R$.

First, the map $\fiblog R [0,\infty)\rightarrow \glone R$ factors through $\glint R$ as is evident from the log LES. 
Then, consider the following diagram, in which the upper vertical maps are the inclusions of the connected covers:
$$\begin{tikzcd}
     \Sigma(\F_\ell^\times)_2\arrow[r]\arrow[d]&\glint[1,\infty)\arrow[r]\arrow[d] &\Sigma(\F_\ell^\times)_2\arrow[d] \\
     \fiblog R[0,\infty)\arrow[r]\arrow[d]&\glint R \arrow[r]\arrow[d]& \glint R[0,1]\arrow[d] \\
     (\F_\ell^\times)_2\rtimes\Z/2\arrow[r]&\Z/2 \arrow[r]&\Z/2 \\
\end{tikzcd}.$$

Since the composition of the top row is an equivalence and the columns are fiber sequences, the outer square of the bottom rectangle is a pullback square. Hence, a section of the bottom horizontal composition provides a section of the map 
$\fiblog R \to \glint R[0,1]$, which then can be composed with the map $\fiblog R \to \glint R$ to provide the desired section of the $1$-truncation map $\glint R \to \glint R[0,1]$. 

Conversely, suppose the truncation map is split. Then since the induced inclusion of $\glint R[0,1]$ into $\glone R$ is null after $T(1)$-localization, it must factor through $\fiblog R$ by its universal property as a fiber. This gives a section $\glint R[0,1] \to \fiblog R[0,\infty)$ which gives on $\pi_0$ the desired splitting of the bottom row.

\end{proof}

By functoriality of the log LES for ring maps, the $\pi_0$-splitting hypothesis in \Cref{splitiffFl} is implied by the $\pi_0$-splitting hypothesis of \Cref{splitiffZ}, which is proved in \Cref{resolveAext}. So assuming these later results for the moment we get the following lemma.
\begin{lem}\label{logfiberKFlp2}
    Let $R= K(\F_\ell)_2$ with $\ell\neq p=2$. Then connective cover of the logarithmic fiber is
    $$\fiblog R[0,\infty)\simeq (\F_\ell^\times)_2\oplus \glint R[0,1].$$
\end{lem}
Note that the homotopy type of the second summand is determined in \Cref{glintsection}. We also get a statement about the splitting of $\glint R$, this time with $p$ and $\ell$ unrestricted

\begin{lem}\label{glintsplitKFlcomplete}
    Let $R=K(\F_\ell)_p$. Then
    $$\glint R\simeq \glint R[0,1]\oplus \glint R[2,\infty).$$
\end{lem}
\begin{proof}
    We have just proved it for $\ell\neq 2=p$. We have proved it for $p>2$ in \Cref{glintsplitKFlpodd}. When $\ell=2=p$, $R=\Z_2$ and the result is trivially true.
\end{proof}

\subsection{For $R=\LKone K(\Z)$}
We have already calculated the connective homotopy groups of $\fiblog R$. As in the case of $K(\F_\ell)$, when $p>2$ there is no possibility of a $k$-invariant.

\begin{lem}\label{logfiberKZoddp}
    Let $R=\LKone K(\Z)$ with $p>2$. Then the connective cover of the logarithmic fiber $\fiblog R$ is
    $$\fiblog R[0,\infty)\simeq \F_p^\times\oplus\Z_p\oplus\Sigma\Z_p.$$
    
\end{lem}

Similarly $\glint R$ is automatically split.
\begin{lem}\label{glintsplitKZpodd}
    Let $R=\LKone K(\Z)$ with $p>2$. Then
    $$\glint R\simeq \Z/2\oplus \Sigma\Z_p\oplus \glint R[2,\infty).$$
\end{lem}
\begin{proof}
    The $\mathbb{Z}/2$ must decouple since the higher homotopy groups are all $p$-complete and $p>2$. The $\Sigma\Z_p$ decouples from the third summand due to the splitting of $\slone$ (see the Key Facts at the beginning of \Cref{section_LT1K(Z)p_odd}).
\end{proof}

At $p=2$ we have an extension problem to solve (cf. \Cref{remA}).

\begin{lem}\label{splitiffZ}
    Let $R=\LKone K(\Z)$ with $p=2$. The truncation map $\glint R\rightarrow \glint R[0,1]$ is split if and only if the group map $A\rightarrow\Z/2$ is split. 
\end{lem}
\begin{proof}
The proof is similar to that of \Cref{splitiffFl}.  The map $\fiblog R[0,\infty)\rightarrow \glone R$ factors through $\glint R$ and we arrive at the following diagram
$$\begin{tikzcd}
     \Sigma(\Z_2\times\Z/2)\arrow[r]\arrow[d]&\glint[1,\infty)\arrow[r]\arrow[d] &\Sigma(\Z_2\times\Z/2)\arrow[d] \\
     \fiblog R[0,\infty)\arrow[r]\arrow[d]&\glint R \arrow[r]\arrow[d]& \glint R[0,1]\arrow[d] \\
     \Z_2\times A\arrow[r]&\Z/2 \arrow[r]&\Z/2 \\
\end{tikzcd}.$$
in which the columns are fiber sequences. The upper composition is an equivalence, so the bottom rectangle is a pullback square. Hence, a section of the composition on the bottom row provides a section of $\fiblog R[0,\infty)\to \glint R[0,1]$, and hence of the truncation map $\glint R \to \glint R[0,1]$. 

Conversely, since there are no non-zero maps from $\glint R[0,1]$ to $R$, a section of the map $\glint R[0,1] \to \glint R$ gives in particular a map $\glint R[0,1] \to \fiblog R [0,\infty)$  which provides a section of the bottom composition. Such a section $s\colon \Z/2 \to \Z_2 \times A$ must have trivial $\Z_2$-component and hence land in $A$. 



\end{proof}

Recall that the group $A$ is one of $\Z/2,\Z/4$, or $\Z/2 \times \Z/2$. We will show now that the last option is the one that occurs, and in particular that the map to $\Z/2$ from \Cref{splitiffZ} indeed splits. 
\begin{lem}\label{resolveAext}
    $A\simeq\Z/2\times \Z/2$.
\end{lem}
\begin{proof}
    The idea is to compare the situation to $KO$, where an analogous extension problem occurs in the log-LES, but where we know the extension splits because we know that $\glone KO$ splits.
    
    Let $R=\LKone K(\Z)$ and consider the map $R\rightarrow \LKone KO=KO_2$. As mentioned in the introduction, a classical result of Adams-Priddy \cite{adamspriddy} says that $\glone KO\simeq \glone KO[0,1]\oplus \glone KO[2,\infty)$. Thus we have $\glint KO_2\simeq \glint KO_2[0,1]\oplus KO_2[2,\infty)$. This implies that $\glint KO_2[0,1]$ is a summand of $\fiblog KO_2$, giving $\fiblog KO_2=\glint KO_2[0,1]\oplus \Z/2$. 

    Now consider the comparison of logarithmic fibers:
    $$\begin{tikzcd}
     \fiblog R[0,\infty)\arrow[r]\arrow[d]&\glint R\arrow[r]\arrow[d] &R\arrow[d] \\
     \fiblog KO_2[0,\infty)\arrow[r]&\glint KO_2 \arrow[r]& KO_2 
\end{tikzcd}.$$
    Within the associated map of log-LES's, we have the following map of short exact sequences
    $$\begin{tikzcd}
     0\arrow[r]& \pi_1\LKone K(\Z)=\Z_2\times\Z/2\arrow[r]\arrow[d]&\Z_2\times A\arrow[r]\arrow[d] &\Z/2=\mathrm{ker}\left(\mathrm{log}_0^{\LKone K(\Z)}\right)\arrow[d]\arrow[r]&0 \\
     0\arrow[r]&\pi_1KO_2=\Z/2\arrow[r]&\Z/2\times\Z/2 \arrow[r]& \Z/2=\mathrm{ker}\left(\mathrm{log}_0^{KO_2}\right) \arrow[r]&0 
\end{tikzcd}.$$
    The right vertical map is an isomorphism as both groups are generated by the class $-1$, and the left vertical map is the projection onto $\Z/2$, as it can be identified with the map 
    \[
    \Z_2 \times \Z/2 \simeq (\Z[1/2]^\times)_2 \to (\R^\times)_2 \iso \Z/2.
    \]
    Finally, the upper left vertical map restricts to an isomorphism between the $\Z_2$-summands in the sense that the composition 
    \[
    \Z_2 \into \Z_2 \times \Z/2 \to \Z_2 \times A \onto \Z_2
    \]
    is an isomorphism. Indeed, since $|A| = 4$, if this composite was not an isomorphism, then the cokernel of the upper left horizontal map has to be of size at least $4$, but it is isomorphic to $\Z/2$ by the exactness of the upper row. 
    Hence, replacing the upper left and upper middle groups with their torsion subgroups keeps the upper row exact. Namely, we have a commutative diagram with exact rows 
 $$\xymatrix{
    0\ar[r]& \Z/2\ar[r]\ar^\wr[d]&A\ar[r]\ar[d] &\Z/2\ar^\wr[d]\ar[r] & 0 \\
    0\ar[r]&  \Z/2\ar[r]&\Z/2\times\Z/2 \ar[r]& \Z/2\ar[r] & 0 
}$$
  so that $A\iso \Z/2 \times \Z/2$ by the $5$-lemma.


\end{proof}

Combining \Cref{splitiffZ} and \Cref{resolveAext} we get the following two results.

\begin{lem}\label{logfiberKZp2}
    Let $R=\LKone K(\Z)$ with $p=2$. Then the connective cover of the logarithmic fiber $\fiblog R$ is
    $$\fiblog R[0,\infty)\simeq \Z_2\oplus\Z/2\oplus \glint R[0,1].$$
    
\end{lem}


From this, we deduce the splitting of $\glint \LKone K(\Z)$:
\begin{cor}\label{glintsplitKZp2}
    Let $R=\LKone K(\Z)$ with $p=2$. Then $\glint R$ splits at $1$, so that we have (cf. Table \ref{tab:log_fib})
    
    $$\glint R \simeq \Z/2\ltimes_{\Sq^2}\Sigma\Z/2 \oplus \Sigma \Z_2\oplus \glint R[2,\infty).$$
\end{cor}

   \begin{proof}
    This follows from the combination of \Cref{splitiffZ}, \Cref{resolveAext}, and \Cref{glintLT1KZchunk}.
   \end{proof}

\subsection{For $R=K(\Z)_p$}\label{sectionKZp}

\begin{lem}
For all $p$, the fiber sequence (cf. \Cref{K1reduction})
$$
    \Z_p\rightarrow \glint (K(\Z)_p)\rightarrow \glint (K(\Z[1/p])_p).
$$
satisfies the property that the first map is null.
\end{lem}
\begin{proof}
    By \Cref{glintsplitKZpodd} and \Cref{glintsplitKZp2} $\glint \LKone K(\Z)$ contains $\Sigma \Z_p$ as a summand. The second map of the above fiber sequence induces an equivalence onto the complementary summand.
\end{proof}

\begin{rem}\label{remarklocsequence44}
    Note that there is also a fiber sequence
    $$\Z_p\rightarrow K(\Z)_p\rightarrow K(\Z[1/p])_p$$ but the first map is not null. Indeed, the third term is the connective cover of a $T(1)$-local ring, and thus cannot contain $\Sigma \Z_p$ as a summand.
\end{rem}

This immediately implies the desired splitting for $\glint K(\Z)_p$.

\begin{cor}\label{glintsplittingpcomplete}
    For every prime $p$ the truncation map
    $$\glint K(\Z)_p\rightarrow\glint K(\Z)_p[0,1]$$ is split. When $p>2$, $\glint K(\Z)_p[0,1]\simeq \Z/2$ and when $p=2$, $\glint K(\Z)_2[0,1]\simeq \glint K(\Z)[0,1]$.
\end{cor}

For later use, we also record here the computation of $\G_m( K(\Z)_p)$. 

\begin{prop} \label{prop:G_m_K(Z)_p}
For $p> 2$ we have 
\[
\G_m(K(\Z)_p) \simeq \F_p^\times \oplus \Z_p  
\]
For $p=2$, we have 
\[
\G_m(K(\Z)_2)\simeq \Z_2 \oplus \Z/2 \oplus \Sigma \Z/2.
\]
In particular, for all $p$ we have $\pi_0\G_m(K(\Z)_p)\simeq \Z_p \times (\Z_p^\times)_{\tors}$.  
Moreover, the map $\pi_0\G_m(K(\Z)_p) \to \pi_0\glone( K(\Z)_p) \simeq \Z_p^\times$ is the zero map for $p=2$ and the projection to the second summand followed by the canonical inclusion $(\Z_p^\times)_{\tors} \into \Z_p^\times$ for odd $p$. 
\end{prop}

\begin{proof}
We have a fiber sequence 
\[
\glone(K(\Z)_p) \to \glone(\LKone K(\Z)) \to \Sigma \Z_p.  
\]
Applying $\hom_{\Sp^\cn}(\Z,-)$ and using the calculations of $\G_m(\LKone K(\Z))$ from \Cref{pistarGmKZK1locoddp} and \Cref{pistarGmKZK1locp2} and the fact that $\hom_{\Sp^\cn}(\Z,\Sigma \Z_p) \simeq \Sigma \Z_p$ we get fiber sequences 
\[
\G_m(K(\Z)_p) \to \Z_p \oplus \F_p^\times \oplus \Sigma \Z_p \to \Sigma \Z_p
\]
for $p$ odd and 
\[
\G_m(K(\Z)_2) \to \Z_2 \oplus \Z/2 \oplus \Sigma \Z_2 \oplus \Sigma \Z/2 \to  \Sigma \Z_2.
\]
We now note that the forgetful map 
\[
\G_m(\LKone K(\Z)) \to \glone(\LKone K(\Z))
\]
sends the topological generator of the $\Z_p$-summand in $\pi_1 \G_m(\LKone K(\Z))$ to the class of $p$ in $\pi_1\glone(\LKone K(\Z))\simeq \pi_1 K(\Z[1/p])_p$; Indeed, this can be seen by inspecting the log LES from the proof of \Cref{pistarGmKZK1locp2} (and similarly for odd primes). This class maps to a generator of $\pi_1\Sigma \Z_p \simeq \Sigma K(\F_p)_p$, as it is the $p$-adic completion of the $p$-adic valuation map. It follows that, in each of the fiber sequences displayed above, the right map is surjective on $\pi_1$, from which the calculation of $\G_m$ follows. 

It remains to compute the map $\pi_0\G_m(K(\Z)_p) \to \pi_0\glone K(\Z)_p \simeq \Z_p^\times$. 
For this, we may use the commutative square 
\[
\xymatrix{
\pi_0\G_m(K(\Z)_p)\ar[r]\ar[d]   & \pi_0\G_m(\LKone K(\Z))\ar[d] \\
\pi_0\glone(K(\Z)_p)\ar^-\sim[r] & \pi_0\glone(\LKone K(\Z))
}
\]
to reduce to the analogous question for $\LKone K(\Z)$. This case, in turn, follows by inspecting the exact sequence associated with the fiber sequence  
\[
\Omega^\infty \G_m(R) \to \glone R \oto{\theta} \Omega^\infty R
\]
presented in \Cref{subsec:strict_units_K(1)_local}.

\end{proof}

\section{The homotopy type of the units of $K(S)$}

\subsection{Splitting the bottom piece of $\glone K(S)$}

\begin{thm}\label{splittingtheoremmain}
    The spectra $\glone K(\Z)$ and $\glone K(\F_\ell)$ are both split at $1$. Namely, the truncation maps
    $$\glone K(\Z)\rightarrow \glone K(\Z)[0,1]$$
    $$\glone K(\F_\ell)\rightarrow \glone K(\F_\ell)[0,1]$$
    admit right inverses.
\end{thm}
\begin{proof}
    Note that $\glone K(\Z)=\glint K(\Z)$. Moreover the homotopy groups of $K(\Z)$ are finitely generated. Having split the truncations $\glint(K(\Z)_p)\rightarrow \glint (K(\Z)_p)[0,1]$ (cf. \Cref{glintsplittingpcomplete}), we can now apply \Cref{arithfracreduction}.

    The splitting for $K(\F_\ell)$ is proved in exactly the same way, using the $p$-complete splittings from \Cref{glintsplitKFlcomplete}.
\end{proof}

\subsection{The homotopy type of $\glone K(S)$}

We can use the splittings in \Cref{splittingtheoremmain} to describe the entire homotopy type of $\glone K(\Z)$ and $\glone K(\F_\ell)$ by identifying the complementary summands $\glone K(\Z)[2,\infty)$ and $\glone K(\F_\ell)[2,\infty)$, at least after completion at a prime $p$. 

\begin{thm}\label{theoremKZ2infty}
    For every prime $p$ there are equivalences 
    $$\glone K(\Z)_p[2,\infty)\simeq K(\Z)_p[2,\infty),$$
    $$\glone K(\F_\ell)_p[2,\infty)\simeq K(\F_\ell)_p[2,\infty).$$
\end{thm}
\begin{proof}
    From the first display of \Cref{sectionKZp} we know that 
    $$\glone K(\Z)_p[2,\infty)\simeq \glone K(\Z[1/p])_p[2,\infty).$$
    Since $K(\Z[1/p])_p$ is the connective cover of its $T(1)$-localization (cf. \Cref{LQprop}) we have a logarithmic identification
    $$\glone K(\Z[1/p])_p[2,\infty)\simeq K(\Z[1/p])_p[2,\infty).$$
    On the other hand, the second fiber sequence in \Cref{sectionKZp} (in \Cref{remarklocsequence44}) implies that $K(\Z[1/p])_p[2,\infty)\simeq K(\Z)_p[2,\infty)$. So we get equivalences 
    $$\glone K(\Z)_p[2,\infty)\simeq K(\Z)_p[2,\infty).$$
    The argument for $K(\F_\ell)$ is analogous.
\end{proof}

\subsection{Relation to strict triviality of $\Sigma^2\Z$}
Consider the element of $\pic(\Z)$ represented by the $\Z$-module $\Sigma^2\Z$. This admits the structure of a strict element of $\pic(\Z)$, as represented by the dashed arrow in the diagram below, where the middle column is a fiber sequence. In fact this strict structure is canonical, since it corresponds to a choice of nullhomotopy of the composite of multiplication by 2 and quotienting by 2. This strict element of $\pic(\Z)$ produces an element $[\Sigma^2\Z]\in\pi_0\G_mK(\Z)$, the composite of the dashed and horizontal arrow in the following diagram
$$\begin{tikzcd}
     & \text{pic}(\Z)\arrow[d]\arrow[r] & \glone K(\Z)\\
\Z\arrow[r,"2"]\arrow[ru,dashed]&\Z\arrow[d, "\Sq^2\circ /2"] & \\
     & \s^2\Z/2& \\
\end{tikzcd}.$$
Note that the horizontal map is an isomorphism on $\pi_1$ and surjective on $\pi_0$. Therefore, if the aforementioned element of $\pi_0\G_mK(\Z)$ were zero, the map $\pic(\Z)\rightarrow \glone K(\Z)$ would factor through the cofiber of the dashed arrow and produce a section of the Postnikov truncation $\glone K(\Z)\rightarrow \glone K(\Z)[0,1]$, providing an alternative proof of \Cref{splittingtheoremmain}. 
Note, however, that the splitting of $\glone K(\Z)$ does not by itself imply the strict triviality of $[\Sigma^2\Z]$: it is not clear that the map $\pic(\Z) \to  \glone K(\Z)$ factors through the summand $\glone K(\Z)[0,1]$. We leave this point for future investigation.

\section{The spectrum of strict units}
In this section we use the arithmetic fracture square to assemble the various facts gathered above into a calculation of the strict units of $K(\Z)$ and $K(\F_\ell)$.

Recall that $\A=\Q\otimes\prod_p\Z_p\subseteq \prod_p \Q_p$ is the ring of finite Adeles. 
\subsection{$\G_mK(\Z)$}
Consider the arithmetic fracture square 
\[
\xymatrix{
\G_m K(\Z) \ar[r]\ar[d]& \prod_p \G_m K(\Z)_p \ar[d] \\ 
\G_m(\Q\otimes K(\Z)) \ar[r] & \G_m(\Q\otimes \prod_p K(\Z)_p ),
}.
\]
For a rational commutative ring spectrum $A$, observe that $\G_mA\to\glone A$ is an equivalence. The homotopy type of the bottom left node is thus completely determined by the homotopy groups of $\Q\otimes K(\Z)$, which is due to Borel (cf. \cite[Proposition 12.2]{Borel}), which are given by $\Q$ in degree 0 and every degree of the form $4k+1$ with $k\geq 1$. Thus the homotopy groups of $\Q\otimes\prod_pK(\Z)_p$ are given by $\A$ in degree 0 and every degree of the form $4k+1$ with $k\geq 1$. Using this and the calculation of the upper right corner in \Cref{prop:G_m_K(Z)_p}, we see that in low degrees we have an exact sequence
\[\begin{tikzcd}[row sep=tiny]
    0\arrow[r]& 0\arrow[r]&0 \\
    \pi_1\G_mK(\Z)\arrow[r]&\Z/2\arrow[r] & 0 \\
    \pi_0\G_mK(\Z)\arrow[r]&\Q^\times\times \prod_p\left(\Z_p\times(\Z_p^\times)_\tors\right)\arrow[r] &\mathbb{A}^\times. \\
\end{tikzcd}\]
To compute $\G_m(K(\Z))$, it thus remains to determine the last map 
\[
\Q^\times\times \prod_p\left(\Z_p\times(\Z_p^\times)_\tors\right)\to \mathbb{A}^\times,
\]
or equivalently, the maps induced from
\[
\prod_p\G_m(K(\Z)_p) \to \G_m(\Q \otimes \prod_p K(\Z)_p) \quad \text{and}\quad \G_m(\Q \otimes K(\Z)) \to \G_m(\Q \otimes \prod_p K(\Z)_p)  
\]
on $\pi_0$. For the right map, since both the source and the target involve only rational ring spectra, $\G_m$ can be replaced by $\glone$ and we get the usual inclusion $\Q^\times \to \A^\times$. 
For the left map, since the target is $\G_m$ of a rational ring spectrum, we can write the map as a composition 
\[
\prod_p\G_m(K(\Z)_p) \to \prod_p\glone (K(\Z)_p) \to \glone(\Q\otimes \prod_p K(\Z)_p) \simeq \G_m(\Q\otimes \prod_p K(\Z)_p), 
\]
so it is enough to compute the map from the strict units of $K(\Z)_p$ to its usual units and the map between units in the middle of the above composition. 

Now, by the second part of \Cref{prop:G_m_K(Z)_p}, for odd prime $p$ the map 
\[
\Z_p \times (\Z_p^\times)_\tors \to \pi_0\glone K(\Z)_p \simeq \Z_p^\times
\]
is zero on the $\Z_p$-summand and the obvious inclusion on the second summand, while for $p=2$ it is the zero map. Hence, the last map in the long exact sequence above is the conglomerate of:
\begin{itemize}
\item The inclusion $\Q^\times \to \A^\times$. 
\item The zero maps $\Z_p \to \A^\times$. 
\item The inclusions $(\Z_p^\times)_{\tors} \into \Z_p^\times \into \Q_p^\times \into \A^\times$ for $p$ odd.
\item The zero map $\Z/2 \simeq (\Z_2^\times)_{\tors} \to \A^\times$. 
\end{itemize}
 Since the map $\Q^\times \times \prod_{p \text{ odd}}(\Z_p^\times)_{\tors} \to \A^\times$ is injective, we deduce that the kernel of the last map in the long exact sequence above is $\Z/2 \times \prod_{p}\Z_p \simeq \widehat{\Z}\times \Z/2$. 


In larger degrees the long exact sequence contains 4-periodically appearing phenomenon ($k\geq 1$):
\[\begin{tikzcd}[row sep=tiny]
    0\arrow[r]& 0\arrow[r]&0 \\
    \pi_{4k+1}\G_mK(\Z)\arrow[r]&\Q\arrow[r] & \A \\
    \pi_{4k}\G_mK(\Z)\arrow[r]&0\arrow[r] &0, \\
\end{tikzcd}\]
Thus we find that $\pi_{4k}\G_mK(\Z)\simeq \A/\Q$.
Since $\G_mK(\Z)$ is a $\Z$-module we have a complete calculation of the spectrum of strict units.
\begin{thm}\label{strictunitsKZ} We have
    $$\G_mK(\Z) \simeq (\widehat{\Z}\times\Z/2)\oplus\Sigma(\Z/2)\bigoplus_{k\geq1}\Sigma^{4k}(\A/\Q).$$
\end{thm}


\subsection{$\G_mK(\F_\ell)$}
From the calculation of the logarithmic fibers (cf. \Cref{logfiberKFloddp} and \Cref{logfiberKFlp2}) we find that $\pi_{\geq2}\G_m(K(\F_\ell)_p)=0$. Note that by Quillen's calculation of $K(\F_\ell)$, its rationalization is concentrated in degree 0 and is isomorphic to $\Q$. Similarly, the rationalization of $\prod_pK(\F_\ell)_p$ is $\A$.  Using \Cref{pistarGmKFlK1loc} and the arithmetic fracture square for $K(\F_\ell)$ we obtain an exact sequence 
$$\begin{tikzcd}[row sep=tiny]
    0\arrow[r]& 0\arrow[r]&0 \\
    \pi_1\G_mK(\F_\ell)\arrow[r]&\prod_p (\F_\ell^\times)_p\arrow[r] & 0 \\
    \pi_0\G_mK(\F_\ell)\arrow[r]&\Q^\times\times \prod_p(\F_\ell^\times)_p\times\F_p^\times\arrow[r] &\A^\times. \\
\end{tikzcd}
$$
Because $[\Z,X]=\pi_0X$ if $X$ is rational, the final map factors through the forgetful map from $\pi_0\G_m$ to $\pi_0\glone $, and is easily seen to be the conglomerate of the inclusion $\Q^\times \into \A^\times$ and the maps $(\F_\ell^\times)_p\times\F_p^\times\rightarrow\Q_p^\times$ that factor through projection to the second factor followed by $\F_p^\times \into \Z_p^\times\hookrightarrow\Q_p^\times$ where the first map is the multiplicative lift. From this we can compute the homotopy groups of $\G_m K(\F_\ell)$ to be $\F_\ell^\times$ in degrees $0$ and $1$. Since $\G_mK(\F_\ell)$ is a $\Z$-module we have a complete calculation of the spectrum of strict units.

\begin{thm}\label{strictunitsKFl}
    $$\G_mK(\F_\ell)=\F_\ell^\times\oplus\Sigma\F_\ell^\times.$$
\end{thm}

\section{The splitting of the sheaf $\glone K$}\label{sectionsheafsplit}

\subsection{Sheaves of Spectra}
Let $\Aff \simeq (\calg_\Z^{\heartsuit})^\op$ be the site of affine schemes\footnote{For set theoretic reasons, one should restrict to affine schemes which are the Zariski spectra of $\kappa$-compact rings for suitable cardinal $\kappa$. Since the $K$-theory functor is accessible, this issue will not cause trouble and we keep it implicit.}, endowed with the Zariski topology. 
Let $\PShv(\Aff;\Sp^\cn)$ be the $\infty$-category of presheaves of connective spectra on $\Aff$, and let $\Shv(\Aff;\Sp^\cn)$ be the full-subcategory of Zariski sheaves. For $\sF\in \PShv(\Aff;\Sp^\cn)$ and a ring $S$, we shall denote $\sF(S):= \sF(\Spec(S))$.  Note that $\Shv(\Aff;\Sp^\cn)$ is a symmetric monoidal localization of $\PShv(\Aff;\Sp)$ via the sheafification functor 
\[
\LZar \colon \PShv(\Aff;\Sp^\cn) \to \Shv(\Aff;\Sp^\cn),
\]
and hence inherits a canonical symmetric monoidal structure. 

Recall that the $\infty$-category $\Shv(\Aff;\Sp)$ has a $t$-structure in which the positive part is $\Shv(\Aff;\Sp^\cn)$. Accordingly, for a sheaf $\sF$ we can set 
$\sF[a,b] := \tau_{\ge a} \tau_{\le b} \sF$, with respect to this $t$-structure. Similarly, we can form the sheaves $\pi_i \sF \in \Shv(\Aff;\Sp^\heartsuit)$, which we regard as discrete objects of $\Shv(\Aff;\Sp^\cn)$. 

\begin{rem}
All the above constructions can be defined explicitly as the sheafifications of the corresponding (pointwise) operations on presheaves. For example, $\sF[a,b]$ is the sheafification of the presheaf given by $A\mapsto \sF(A)[a,b]$, e.t.c.
\end{rem}

We can thus define the analog of \Cref{def:split_at} in the context of sheaves. 

\begin{defn}\label{def:split_at_sheaf}
We say that a sheaf of connective spectra $\sF \in \Shv(\Aff;\Sp^\cn)$ \emph{splits at $a$} if $\sF \simeq \sF[0,a] \oplus \sF[a+1,\infty)$. 
\end{defn}

\begin{warn}
We warn the reader right away that since sheafification is involved in the formation of Postnikov truncations of sheaves, the fact that a sheaf $\sF$ splits at $a$ \emph{does not imply} that $\sF(S)$ splits at $a$ for every ring $S$. In fact, we will produce a concrete counterexample of this phenomenon in \Cref{pointwisecounterexample}. 
\end{warn}

\subsection{$K$-theory sheaf}

The goal of this section is to show that the spectrum of units of $K$-theory, considered as a Zariski sheaf, splits at $1$. We begin by introducing this sheaf.
Since algebraic $K$-theory satisfies Zariski descent \cite{thomasontrobaugh}, we have a commutative algebra object $K\in \calg(\Shv(\Aff;\Sp^\cn))$ (where we consider here the connective $K$-theory for convenience).
Taking units produces a sheaf of connective spectra 
$
\glone K \in \Shv(\Aff;\Sp^\cn)
$
by the formula 
\[
(\glone K)(S) := \glone(K(S)).
\] 
\begin{rem}
Note that since $K$-theory preserves filtered colimits, the stalks of the $K$-theory sheaf at the points of $\Spec(S)$ are the $K$-theory spectra of the localizations of $S$. Hence, the same holds for $\glone K$.   
\end{rem}

For a connective spectrum $X$, let $\ct{X}$ denote the constant Zariski sheaf on $X$, that is, the sheafification of the constant presheaf with value $X$. Then for each $\sF\in 
\PShv(\Aff;\Sp^\cn)$, there is a natural identification 
\[
\Map(\ct{X},\sF) \simeq \Map(X,\sF(\Z)),
\]
exhibiting the functor $X\to\ct{X}$ as left adjoint to the evaluation-at-$\Z$ functor $\PShv(\Aff;\Sp^\cn)\to \Sp^\cn$.

\begin{lem} \label{lem:homotopy_sheaves_glone}
The lowest homotopy sheaves of $\glone K$ are given by 
\[
\pi_0 \glone K \simeq \ct{\Z/2} \quad \text{and}\quad \pi_1 \glone K \simeq \G_m,
\]
where $\G_m$ here stands for the sheaf of units on $\Aff$.  
\end{lem}

\begin{proof} 
We start with $\pi_1$.
Let $\pic\in \Shv(\Aff;\Sp^\cn)$ be the Picard sheaf, taking $S$ to invertible objects of $\Mod_S(\Sp)$. The maps $\pic(S) \to \glone K(S)$ are natural in $S$ and hence assemble to a morphism of sheaves $\pic \to \glone K$. Taking loops, we get a map $\G_m \simeq \Omega \pic \to \Omega \glone K$. We claim that this map induces an isomorphism on $\pi_0$-sheaves. Indeed, this can be checked at all the stalks, and since both functors preserve filtered colimits, it thus suffices to check this at local rings. The claim then follows from the fact that $K_1(S) \simeq S^\times$ when $S$ is a local ring.   

For $\pi_0$, we have a canonical map 
\[
\Z/2 \iso \pi_0 (\glone K(\Z)) \to (\pi_0\glone K)(\Z) 
\]
which corresponds to a map $\ct{\Z/2} \to \pi_0\glone K$. To see that this map is an isomorphism, we may again restrict to local rings, where the claim follows from the fact that $\pi_0 K(S) \simeq \Z$ when $S$ is a local ring. 

\end{proof}

\begin{lem} \label{glone_sheaves_pushout}
There is a pushout square in $\Shv(\Aff;\Sp^\cn)$ of the form 
\[
\xymatrix{
\ct{\Sigma \pi_1 \glone K(\Z)} \ar[r]\ar[d]  & \Sigma \pi_1 \glone K \ar[d] \\
\ct{\glone K(\Z) [0,1]}\ar[r] & \glone K [0,1].  
}
\]
\end{lem}

\begin{proof}
First, the existence of the commutative square follows immediately from the universal property of the constant sheaf construction. In other words, there is a natural commutative square of functors in $S$ of the form 
\[\xymatrix{
\Sigma \pi_1 \glone K(\Z) \ar[r]\ar[d] & \Sigma \pi_1 \glone K(S) \ar[d]  \\ 
\glone K(\Z)[0,1] \ar[r] & \glone K(S)[0,1]   
}
\]
which sheafifies to the square in the lemma. 

To show that our square is a pushout, it is enough to show that the map between the vertical cofibers is an isomorphism. Since the vertical maps are the connected covers of the $1$-truncated sheaves in the bottom row, this is equivalent to the fact that the bottom horizontal map induces an isomorphism
on $\pi_0$, which in turn follows from \Cref{lem:homotopy_sheaves_glone}.  

\end{proof}

We can now produce the splitting of sheaf $\glone K$. 

\begin{thm}\label{sheafsplit_main}
The sheaf $\glone K$ splits at $1$. Namely, there is an isomorphism $\glone K \simeq \glone K[0,1] \oplus \glone K [2,\infty)$ of sheaves of connective spectra on $\Aff$.  
\end{thm}

\begin{proof}
We will construct a map $\glone K[0,1] \to \glone K$ such that the composition $\glone K[0,1]\to \glone K \to \glone K[0,1]$ is an isomorphism. By \Cref{glone_sheaves_pushout}, producing the map amounts to producing a commutative square of the form 
\[
\xymatrix{
\ct{\Sigma \pi_1 \glone K(\Z)} \ar[r]\ar[d]  & \Sigma \pi_1 \glone K \ar[d] \\
\ct{\glone K(\Z) [0,1]}\ar[r] & \glone K.  
}
\]
with the same left vertical and upper horizontal maps as in \Cref{glone_sheaves_pushout}. In other words, we have to construct maps 
\[
\Sigma \pi_1 \glone K \too \glone K
\quad \text{and} \quad  \ct{\glone K(\Z) [0,1]} \too \glone K
\] 
together with a homotopy between the resulting pair of maps $\ct{\Sigma \pi_1 \glone K(\Z)} \to \glone K$. 

\begin{itemize}
\item To construct the right vertical map, we proceed as in \Cref{splitting_slone}. Namely,
consider the map $\pic \to \glone K$; taking loops, we get a map $\pi_1\glone K \simeq \pi_1 \pic \to \Omega \glone K$ whose adjoint is a map $\Sigma \pi_1 \glone K \to \glone K$ as desired. 

\item To construct the bottom horizontal map, by the adjunction between the constant sheaf functor and evaluation-at-$\Z$, it suffices to construct a map $\glone K(\Z)[0,1] \to \glone K(\Z)$. Such a map is provided by the splitting of $\glone K(\Z)$ in \Cref{splittingtheoremmain}. 

\item To show that the resulting maps from $\ct{\Sigma \pi_1 \glone K(\Z)}$ to $\glone K$ are homotopic, we observe that 
\begin{align*} 
\Map(\ct{\Sigma \pi_1 \glone K(\Z)},\glone K) & \simeq \Map(\Sigma \pi_1 \glone K(\Z),\glone K(\Z)) \\ &\simeq  \Map(\Sigma \Z/2,\glone K(\Z)) \\   & \simeq\Map_{\Z}(\Sigma\Z/2,\G_m K(\Z)) \simeq \Z/2, 
\end{align*}
where the first isomorphism is by the constant sheaf adjunction, the second by the fact that $K_1(\Z) \simeq \Z/2$, the third by the adjunction between $\hom(\Z,-)$ and the forgetful functor from $\Z$-modules to spectra, and the last by the computation of $\G_m K(\Z)$ in \Cref{strictunitsKZ}. 
Unwinding these identifications, we see that a map $\alpha \colon \ct{\Sigma \pi_1 \glone K(\Z)} \to \glone K$ is determined by the composition 
\[
\Sigma \Sph \to \Sigma \Z/2 \simeq (\ct{\Sigma \pi_1 \glone K(\Z)})(\Z) \oto{\alpha(\Z)} \glone K(\Z).  
\]
It is now easy to check that for the constructions above, these composites both give the non-trivial class in $\pi_1 \glone K(\Z) \simeq \Z/2$, and hence they are homotopic. 

\end{itemize}

Finally, to get the splitting of $\glone K$ from the map $\glone K[0,1] \to \glone K$, we have to check that it induces isomorphisms on the $\pi_0$- and $\pi_1$-sheaves. Note that in the square of \Cref{glone_sheaves_pushout}, the left vertical map is an isomorphism on $\pi_1$-sheaves and the lower horizontal map is an isomorphism on $\pi_0$-sheaves. Hence, the result follows from the facts that the maps $\ct{\glone K(\Z)[0,1]} 
 \to \glone K$ and $\Sigma \pi_1 \glone K \to \glone K$ in the square from the beginning of the proof also induce isomorphisms on $\pi_0$ and $\pi_1$ respectively, which is clear by \Cref{lem:homotopy_sheaves_glone}.   
\end{proof}





We mention again that the result above does not imply that $\glone K(S)$ splits at $1$ for all $S$ (we disprove this in the following section). Instead, the spectrum $\glone K(S)$ contains the sections of the \emph{sheaf} $\glone K[0,1]$ over $\Spec(S)$ as a summand. On the positive side, in the cases where the sections of $\glone K[0,1]$ coincide with the $1$-truncation of $\glone K(S)$, we do get a splitting for $\glone K(S)$ at $1$. 
We finish by spelling out this condition a bit more explicitly and give a few examples. 

For a commutative ring $S$ let $\hPic(S)$ be the Picard group of the abelian category of discrete $S$-modules. 
Recall that there is a map $K_0(S) \to \hPic(S) \oplus \ct{\Z}(S)$
taking a virtual projective $S$ module to its determinant line bundle and its (locally constant) rank function, respectively. 
Similarly, in degree $1$ there is a determinant map 
$K_1(S) \to S^\times$. 

\begin{prop}\label{low_K_correct_splitting}
Let $S$ be a commutative ring for which the maps $K_0(S) \to \hPic(S) \oplus \ct{\Z}(S)$ and $K_1(S) \to S^\times$ are isomorphisms. Then $\glone K(S)$ splits at $1$.  
\end{prop} 

\begin{proof}
By \Cref{sheafsplit_main} it would suffice to show that the canonical map $\glone K(S)[0,1] \to (\glone K [0,1])(S)$ is an isomorphism. This map is obtained from the map of commutative ring spectra $K(S)[0,1] \to K[0,1](S)$ by applying the functor $\glone(-)$, so it suffices to show that the ring map is an isomorphism. 

Now, we have a fiber sequence
\[
\Sigma \G_m \too K[0,1] \too \ct{\Z} 
\]
and the sheaf $\Sigma \G_m$ has homotopy groups  
\[\pi_0((\Sigma \G_m)(S)) \simeq \hPic(S)\quad \text{and} \quad \pi_1(\Sigma \G_m)(S) \simeq S^\times.  
\]
Thus, we obtain a canonical isomorphism $\pi_1 (K[0,1](S)) \simeq S^\times$ and an exact sequence
\[
0 \too \hPic(S) \too \pi_0  (K[0,1](S)) \too \ct{\Z}(S). 
\]
This sequence is in fact split by the map $\ct{\Z}(S) \to K(S) \to K[0,1](S)$ that picks the trivial modules of given (locally constant) rank, so that $\pi_0 (K[0,1](S))\simeq \ct{\Z}(S) \oplus \hPic(S)$. 
It is straightforward to check that via these isomorphisms the maps $\pi_i K(S) \to \pi_i K[0,1] (S)$ for $i=0,1$ correspond  precisely to the maps from the statement of the proposition, so the result follows from our assumptions on $S$.    
\end{proof}

There are many rings satisfying the conditions. 
\begin{exmp}
Every local ring and every Euclidean domain satisfies the conditions of \Cref{low_K_correct_splitting} (cf.  e.g. \cite[Corollary 2.6.3 \& Example 1.3.5]{K-book}). Hence, for these rings we obtain natural splittings
\[
\glone K(S) \simeq \glone K(S)[0,1] \oplus \glone K(S)[2,\infty).
\]
\end{exmp}

\section{A counterexample to pointwise splitting of $\glone K$}

Observe that $\glone K$ does not split as a Zariski \emph{presheaf}. Indeed, such a splitting would imply that the presheaf
\[
S \longmapsto \glone K(S)[0,1]
\]
is a summand of the sheaf $\glone K$, and hence is itself a sheaf, which it clearly is not. One may still ask whether $\glone K(S)$ splits for every ring $S$, albeit nonfunctorially. In this section, we show that even this is false: we construct a ring $S$ for which the map
\[
\glone K(S)\longrightarrow \glone K(S)[0,1]
\]
does not split. Our starting point is the following topological nonsplitting result for the $\KO$-valued cohomology of the circle.

\begin{prop}\label{prop:KO_circ_not_split}
The spectrum $\glone \KO(S^1)[1,\infty]$ does not split at $1$. In other words, there is no splitting of the form
\[
\glone \KO(S^1)[1,\infty]
\simeq
\Sigma(\Z/2\oplus\Z/2)\oplus \glone \KO(S^1)[2,\infty].
\]
\end{prop}

\begin{proof}
Recall that $\glone\KO(S^1)$ is the connective cover of
\[
(\glone\KO)^{S^1}
\simeq
\glone\KO\oplus\Sigma^{-1}\glone\KO.
\]
Thus, a splitting as in the statement would induce two splittings: one of
$\glone\KO[1,\infty]$ at $1$, and another of
$\glone\KO[2,\infty]$ at $2$.

We claim that the latter splitting does not exist. Indeed, after passing to
$2$-adic completion, such a splitting would yield a corresponding splitting
of $\glone\KO_2[2,\infty]$. Consequently, $\pi_2\G_m(\KO_2)$ would contain
\[
\pi_2\glone (\KO_2)\simeq\Z/2
\]
as a direct summand, contradicting \Cref{ahrz}.
\end{proof}

To deduce from this an example of a ring $S$ for which $\glone K(S)$ does
not split at $1$, we take $S$ to be the ring of functions on an affine
algebraic variety $X$ over $\R$ whose real locus is $S^1$.
Indeed, for such $X$, restriction of algebraic vector bundles to the
real locus defines a map of commutative ring spectra
\[
r\colon K(X)\longrightarrow \KO(X(\R))\simeq\KO(S^1)
\]
(cf.~\cite[Section 1]{KaroubiWeibel2003}). This, in turn, induces a map of spectra
\[
\glone K(X)\longrightarrow\glone\KO(S^1).
\]
We now take
\[
X=\Spec\bigl(\R[x,y]/(x^2+y^2-1)\bigr),
\]
which we may, and will, identify with the open subvariety of $\P^1_\R$
obtained by removing the closed point corresponding to the Galois orbit
$\{[1:i],[1:-i]\}$.

Let $\varepsilon\in\KO^1(S^1)$ denote the canonical element, so
that $\eta\varepsilon$ is the difference between the M\"obius bundle and
the trivial line bundle. For $a\in\R^\times$, set
\[
\sign(a):=
\begin{cases}
0, & a>0,\\
1, & a<0.
\end{cases}
\]

\begin{prop}[{{cf.~\cite[Appendix A]{KaroubiWeibel2003}}}]
There is an isomorphism
\[
K_1(X)\simeq\R^\times\oplus\Z/2
\]
under which the map
\[
r\colon K_1(X)\longrightarrow\pi_1\KO(S^1)
\]
is given by
\[
(a,b)\longmapsto \sign(a)\eta+b\varepsilon\eta^2.
\]
In particular, $r$ restricts to an isomorphism
\[
K_1(X)_\tors\xrightarrow{\;\sim\;}\pi_1\KO(S^1)
\]
from the torsion subgroup of $K_1(X)$ onto $\pi_1\KO(S^1)$.
\end{prop}

\begin{proof}
Since $X$ is the open subvariety of $\P^1_\R$ whose complement is
$\Spec(\C)$, and since
\[
K(\P^1_\R)\simeq K(\R)\oplus K(\R),
\]
the localization sequence gives a cofiber sequence
\[
K(\C)\longrightarrow K(\R)\oplus K(\R)\longrightarrow K(X).
\]
More precisely, the middle spectrum is freely generated as a
$K(\R)$-module by the classes
\[
1=[\cO_{\P^1_\R}]
\qquad\text{and}\qquad
\xi=[\cO_{\P^1_\R}(1)]-[\cO_{\P^1_\R}].
\]
With respect to this basis, the map
\[
K(\C)\longrightarrow K(\R)\oplus K(\R)
\]
is given by $(0,\tr)$, where
\[
\tr\colon K(\C)\longrightarrow K(\R)
\]
is induced by restriction of scalars along $\R\to\C$.

Taking the associated long exact sequence of homotopy groups, and using
the fact that $\tr$ induces multiplication by $2$ on $K_0$ and the norm
map $\C^\times\to\R^\times$ on $K_1$, we obtain
\[
\C^\times
\oto{(0,\operatorname{Nm}^{\C}_{\R})}
\R^\times\oplus\R^\times
\longrightarrow K_1(X)
\longrightarrow
\Z
\oto{(0,2)}
\Z\oplus\Z
\longrightarrow 0.
\]
It follows that
\[
K_0(X)\simeq\Z\oplus\Z/2
\]
and
\[
K_1(X)
\simeq
\R^\times\oplus
\R^\times/\operatorname{Nm}^{\C}_{\R}(\C^\times)
\simeq
\R^\times\oplus\Z/2,
\]
as claimed.

It remains to determine the map
\[
r\colon K_1(X)\longrightarrow\pi_1\KO(S^1).
\]
With respect to the basis $1,\xi$ of $K(\P^1_\R)$, the generators of
$K_0(X)$ and $K_1(X)$ appearing above may be described as follows:
\begin{itemize}
\item The $\Z$-summand of $K_0(X)$ is generated by $1$.
\item The $\Z/2$-summand of $K_0(X)$ is generated by the restriction of $\xi$ to $X$, which we denote again by $\xi$.
\item The class $[-1]\in\R^\times\subseteq K_1(X)$ is the Hopf element $\eta$.
\item The $\Z/2$-summand of $K_1(X)$ is generated by $\eta\xi$.
\end{itemize}
These descriptions follow from the localization sequence together with
the fact that the Hopf element $\eta\in K_1(\R)$ is represented by the unit
$-1\in \R^\times \simeq K_1(\R)$, as explained in the proof of \Cref{lemmasq2glint}.

Finally, the real vector bundle on $S^1$ obtained by restricting
$\cO_{\P^1_\R}(1)$ to the real locus is the M\"obius bundle. Hence
\[
r(\xi)=r([\cO_{\P^1_\R}(1)]) - r([\cO_{\P^1_\R}]) = (1 + \varepsilon \eta)  - 1 =  \varepsilon\eta.
\]
From this we deduce that 
\[
r(\eta\xi)
=
\eta\,r(\xi)
=
\eta^2\varepsilon.
\]
Next, the restriction of $r$ to $\R^\times_{>0}$ is trivial: indeed,
$\R^\times_{>0}$ is divisible, whereas $\pi_1\KO(S^1)$ is $2$-torsion.
Thus, for every $a\in\R^\times\subseteq K_1(X)$,
\[
r(a)
=
r\bigl((-1)^{\sign(a)}\bigr)
=
r(\sign(a)\eta) = \sign(a)\eta
\]
since $r$ is a morphism of ring spectra. 
This proves the result.
\end{proof}

\begin{cor}\label{pointwisecounterexample}
The spectrum $\glone K(X)$ does not split at $1$.
\end{cor}

\begin{proof}
Suppose, for contradiction, that there is a splitting
\[
s\colon \glone K(X)[0,1]\longrightarrow \glone K(X).
\]
Restricting $s$ to the torsion subgroup of $\pi_1K(X)$ and composing with
the realization map gives a map
\[
\Sigma\pi_1\KO(S^1)
\simeq
\Sigma K_1(X)_\tors
\longrightarrow
\glone K(X)[1,\infty]
\oto{r}
\glone\KO(S^1)[1,\infty]
\]
which induces an isomorphism on $\pi_1$.
Thus, this map gives a splitting of $\glone\KO(S^1)[1,\infty]$ at $1$,
contradicting \Cref{prop:KO_circ_not_split}.
\end{proof}

\bibliographystyle{plain}

\bibliography{references}

\end{document}